\documentclass[12pt]{amsart}
\usepackage{amssymb, eucal, amsfonts, amsmath, xy,latexsym, bbm}
\usepackage{mathrsfs, minibox}

\usepackage{amscd, bbm, amsaddr}
\usepackage{amssymb,mathabx}
\usepackage{mathrsfs,wasysym, tikz-cd, extpfeil}
\usetikzlibrary{matrix, arrows}
\usepackage{amsmath}
\usepackage[margin=2cm]{geometry}

\usepackage[margin=2cm]{geometry}
\usepackage[pdftex,colorlinks,hypertexnames=false,linkcolor=black,urlcolor=black,citecolor=black]{hyperref}

\allowdisplaybreaks

\numberwithin{equation}{section}

\newtheorem{Theorem}{Theorem}[section]
\newtheorem*{Theorem*}{Theorem}
\newtheorem*{Corollary*}{Corollary}

\newtheorem{Lemma}[Theorem]{Lemma}
\newtheorem{Proposition}[Theorem]{Proposition}
\newtheorem{Corollary}[Theorem]{Corollary}

\newtheorem*{Theorem A}{Theorem A}
\newtheorem*{Theorem D}{Theorem D}
\newtheorem*{Theorem C}{Theorem C}
\newtheorem*{Theorem B}{Theorem B}

\theoremstyle{definition}

\theoremstyle{remark}
\newtheorem{Remark}[Theorem]{Remark}
\newtheorem*{Remark*}{Remark}

\newcommand{\C}{\mathbb{C}}

\newcommand{\Z}{\mathbb{Z}}

\renewcommand{\Z}{\mathbb{Z}}

\newcommand{\g}{\mathfrak{g}}
\newcommand{\gl}{\mathfrak{gl}}
\renewcommand{\sl}{\mathfrak{sl}}
\newcommand{\so}{\mathfrak{so}}
\renewcommand{\sp}{\mathfrak{sp}}

\newcommand{\cc}{\mathfrak{c}}

\newcommand{\h}{\mathfrak{h}}

\renewcommand{\u}{\mathfrak{u}}
\newcommand{\af}{\mathfrak{a}}

\newcommand{\mi}{\mathfrak{i}}

\renewcommand{\P}{\mathcal{P}}
\renewcommand{\O}{\mathbb{O}}

\newcommand{\X}{\mathcal{X}}

\newcommand{\Y}{\mathcal{Y}}

\newcommand{\ad}{\operatorname{ad}}

\newcommand{\Lie}{\operatorname{Lie}}

\newcommand{\Hom}{\operatorname{Hom}}
\newcommand{\gr}{\operatorname{gr}}
\newcommand{\Spec}{\operatorname{Spec}}

\newcommand{\Mat}{\operatorname{Mat}}

\newcommand{\Aut}{\operatorname{Aut}}

\newcommand{\GL}{\operatorname{GL}}

\newcommand{\SO}{\operatorname{SO}}

\newcommand{\F}{\operatorname{F}}

\newcommand{\ve}{\varepsilon}

\renewcommand{\h}{\widehat}

\newcommand{\he}{\widehat{e}}
\newcommand{\hd}{\widehat{d}}
\newcommand{\hcc}{\widehat{\cc}}

\newcommand{\ce}{\widecheck{e}}
\newcommand{\cd}{\widecheck{d}}
\newcommand{\cy}{\widecheck{y}}
\newcommand{\td}{\widetilde{d}}

\newcommand{\teta}{\widetilde{\eta}}

\newcommand{\ttheta}{\dot{\theta}}
\newcommand{\od}{\dot{d}}
\newcommand{\bnu}{\bar{\nu}}

\newcommand{\geez}{\operatorname{\ge\!0}}

\newcommand{\isoto}{\overset{\sim}{\longrightarrow}}
\newcommand{\onto}{\twoheadrightarrow}

\DeclareRobustCommand\longtwoheadrightarrow
     {\relbar\joinrel\twoheadrightarrow}
\mathchardef\mh="2D

\title[]{\boldmath Shifted twisted Yangians and \\  Slodowy slices in classical Lie algebras}

\author{Lukas Tappeiner \& Lewis Topley}

\thanks{\nonumber{\it Mathematics Subject Classification} (2000 {\it revision}).
Primary 17B63, 17B37. Secondary 17B45, 17B08.}

\begin{document}
\maketitle

\begin{abstract}
In this paper we introduce the shifted twisted Yangian of type {\sf AI}, following the work of Lu--Wang--Zhang, and we study their semiclassical limits, a class of Poisson algebras. We demonstrate that they coincide with the Dirac reductions of the semiclassical shifted Yangian for $\gl_n$. We deduce that these shifted twisted Yangians admit truncations which are isomorphic to Slodowy slices for many non-rectangular nilpotent elements in types {\sf B}, {\sf C}, {\sf D}. As a direct consequence we obtain parabolic presentations of the semiclassical shifted twisted Yangian, analogous to those introduced by Brundan--Kleshchev for the Yangian of type {\sf A}. Finally we give Poisson presentations of Slodowy slices for all even nilpotent elements in types {\sf B}, {\sf C}, {\sf D}, generalising the recent work of the second author.
\end{abstract}

\section{Introduction}

\subsection{Background}

To each nilpotent orbit in the Lie algebra of a complex reductive group one can associate a Slodowy slice, which is a certain affine Poisson variety, and these slices are quantized by finite $W$-algebras \cite{GG02, Pr02}. These slices and their quantizations play several interesting roles in representation theory, algebraic geometry and mathematical physics.

The representation theory of the finite $W$-algebra has furnished new methods for studying the invariants of primitive ideals of enveloping algebras \cite{Br11, Lo10a, Lo15, PrJI, Pr10, Pr14, PT14}, as well as the quantizations of nilpotent orbits and their covers \cite{Lo10b, Lo22, To23}.

At the same time, a Slodowy slice can be viewed as a flat deformation of its nilpotent part \cite[Corollary~7.4.1]{Sl80}, and the latter is an example of a conic symplectic singularity \cite[Proposition~2.1.2]{Gi09}. Recently it has been conjectured that these nilpotent slices are the symplectic duals of certain affinizations of covers of nilpotent orbits \cite[\textsection~9.3]{LMM21}, and this duality is the central motif in the field of symplectic representation theory; see \cite{BLPW16} for an introduction.

Finite $W$-algebras and Slodowy slices also play a vital role in the theory of vertex algebras. Affine $W$-algebras form a large class of vertex algebras, which have been studied intensively in recent years \cite{Ar15, Ar17, AF19, DSKV16, Fr02, AKMFP24}. By taking the Zhu algebra or the $C_2$-algebra of an universal affine $W$-algebra at any level, one obtains a finite $W$-algebra or a Slodowy slice, respectively \cite[\textsection 5.7]{Ar17}. In this way, the algebras discussed in the present paper are the key tools for studying conformally graded simple modules and associated varieties of the simple affine vertex algebras. These are some of the key challenges in this field.

Yangians are a family of associative algebras of a very different pedigree.  
They were first introduced by the St. Petersburg school in the 1980s and are one of the fundamental examples of quantum groups defined using the RTT formalism \cite{FRT90}. Drinfeld made an extensive study of the representation theory of the Yangian associated to $\gl_n$, and introduced a new presentation to classify their finite dimensional modules \cite{Dr87, Dr88}. Later Olshanski introduced twisted Yangians, which are analogous to the twisted affine Lie algebras \cite{Ol92}, one of the key ingredients in the current article.

In \cite{RS99} Ragoucy and Sorba showed that if $N = nl$ for $n,l > 0$ and we choose the nilpotent orbit $\O \subseteq \gl_N$ which has all blocks of size $l$ (the so-called {\it rectangular} case) then the finite $W$-algebra associated to $\O$ can be realised as a quotient of a Yangian associated to $\gl_n$. Brundan and Kleshchev extended this relationship between Yangians and $W$-algebras, by introducing the shifted Yangians \cite{BK06}. They showed that for every $N>0$ and to every nilpotent orbit $\O \subseteq \gl_N$ we can associate a {\it shift matrix} $\sigma$ and a {\it  level} $l$, and that the truncated shifted Yangian $Y_{n,l}(\sigma)$ is isomorphic to the finite $W$-algebra for $\O$. The (doubled) canonical filtration on $Y_{n,l}(\sigma)$ lines up with the Kazdan filtration on the finite $W$-algebra and so, as a corollary, the semiclassical truncated shifted Yangian is isomorphic to the Slodowy slice, as a Poisson algebra. One of the key features of their argument was the introduction of {\it parabolic presentations} of Yangians \cite{BK05}, which we discuss later in this article.

In types {\sf B}, {\sf C} and {\sf D} much less is known. Ragoucy showed that certain truncations of semiclassical twisted Yangians are isomorphic (as Poisson algebras) to the Slodowy slices associated to rectangular nilpotent orbits \cite{Ra01}. Brown proved the quantum analogue for finite $W$-algebras using quite different methods \cite{Br09}. There has been a long-standing hope that some variant of  twisted Yangians could be related to Slodowy slices and finite $W$-algebras associated to non-rectanguar cases. In this paper we make the first progress in this problem by introducing shifted twisted Yangians of type {\sf AI}, and relating them to Slodowy slices. Our main tools are the Drinfeld presentation of twisted Yangians recently introduced by Lu, Wang and Zhang \cite{LWZ23}, as well as the methods of \cite{To23} and the parabolic presentations of Yangians introduced by Brundan--Kleshchev \cite{BK05}.

\subsection{Statement of results}

We introduce some notation related to Yangians, and we refer the reader to \cite{Mo07} for a good introduction. Fix $n > 0$, let $Y_n$ denote the Yangian associated to $\gl_n$ and let $y_n$ denote the (semiclassical) Yangian $\gr Y_n$, with respect to the canonical filtration. Since this algebra is commutative it carries a natural Poisson structure, and because this paper focuses on $y_n$ rather than $Y_n$, we will usually drop the prefix `semiclassical' in what follows.

Pick a symmetric shift matrix $\sigma \in \Mat_n(\Z_{\ge 0})$ (see \textsection \ref{ss:symmetricshift}) and let $y_n(\sigma)$ be the shifted Yangian, which is presented by generators and relations in \cite[\textsection 3.3]{To23}. According to Lemma~\ref{L:lambdabij} there is a natural bijection between pairs $(\sigma, \ell)$, where $\ell>0$ is an integer we call the {\it level}, and partitions of positive integers with precisely $n$ parts, all of the same parity.

Let $\Y_n$ denote the (non-commutative) twisted Yangian for the orthogonal Lie algebra $\mathfrak{o}_n$, where the twisting is performed with respect to the Cartan involution. Note that $\Y_n$ is denoted $Y_G(\mathfrak{o}_N)$ in \cite[Ch. 2]{Mo07}, where $G$ is the identity matrix in $\Mat_n(\C)$. Let $y_n^+$ denote the semiclassical limit of $\Y_n$ with respect to the canonical filtration, which we equip with its natural Poisson structure. In Section~\ref{ss:shiftedtwisted} we introduce the shifted twisted Yangian $\Y_n(\sigma)$ as well as the Poisson algebra $y_n^+(\sigma) = \gr \Y_n(\sigma)$.

In Section~\ref{ss:Diracforfixed} we recall the method of Dirac reduction for complex affine Poisson algebras. If $A$ is a countably generated complex Poisson algebra with a reductive group $H$ of Poisson automorphisms then $R(A,H)$ is a Poisson algebra such that $\Spec R(A,H) \cong X^H$. The Cartan involution on $\gl_n$ induces an involution on $y_n$ which we denote by $\tau$, see \eqref{e:tauon1}. Since $\sigma$ is symmetric the involution restricts to $y_n(\sigma)$.

Our first main theorem is the following, see Theorem~\ref{yangianquantizes}.
\begin{Theorem A}
 We have a Poisson isomorphism
 $$R(y_n(\sigma), \tau) \isoto y_n^+(\sigma).$$
\end{Theorem A}
The theorem is proven by defining a map on generators, checking relations, and then applying the Poincar{\'e}--Birhoff--Witt (PBW) theorem on both sides to see that a basis maps to a basis. We prove a PBW theorem for $\Y_n(\sigma)$, drawing heavily on the unshifted case (see \cite[\textsection 5.4]{LWZ23}) using an argument analogous to \cite[Theorem~2.1]{BK06}, and we deduce the PBW theorem for $y_n^+(\sigma)$ from the non-commutative setting. The PBW theorem for $R(y_n(\sigma), \tau)$ is \cite[Theorem~3.4]{To23}.

The main result of \cite[Part I]{To23} states that $R(y_n(\sigma), \tau)$ admits certain truncations which are isomorphic to Slodowy slices as Poisson algebras. Consequently we obtain the following.
\begin{Corollary*}
Suppose that either:
\begin{itemize}
\item[(i)] $\g = \so_n$ and all parts of the partition are odd; or
\item[(ii)] $\g = \sp_n$ and all parts of the partition are even.
\end{itemize}
Then the Poisson structure on the Slodowy slice is isomorphic to a truncation of the shifted twisted Yangian for $\mathfrak{o}_n$.
\end{Corollary*}

One of the key tools in the proofs of \cite{BK06} are the parabolic presentations of shifted Yangians. These interpolate between the RTT presentation and the Drinfeld presentation. The former should be thought of as analogous to presenting a simple Lie algebra via a Chevalley $\Z$-form, whilst the latter is akin to the Chevalley--Serre presentation. For future developments it is important to discover parabolic presentations of twisted Yangians $\Y_n$ and their shifted analogues (see \cite[\textsection 1.3]{LWZ23}). Our second main theorem provides these presentations at the semiclassical level. We state the result informally now, and provide a more detailed statement and proof in Remark~\ref{Proof of thm B}.
\begin{Theorem B}
We introduce parabolic presentations of $y_n^+(\sigma)$.
\end{Theorem B}

In the Corollary to Theorem~A we discussed Poisson presentations of a certain class of even orbit in types {\sf B}, {\sf C} and {\sf D}. Our final main theorem extends this presentation to all even orbits, by using the parabolic presentations of semiclassical shifted Yangians, generalising \cite[Part~I]{To23}. The key difference between Theorems~A and C is that we cannot yet relate the new Poisson presentations below to a shifted twisted Yangian.

The full statement and proof of the following theorem can be assembled from Theorem~\ref{dirredyangpres} and Theorem~\ref{kernelslicepres}.
\vspace{4pt}

\begin{Theorem C}
A presentation of the Poisson structure on the Slodowy slice associated to any even orbit in $\so_n$ or $\sp_n$ is given by the truncation of the Dirac reduction of the semiclassical shifted Yangian, with respect to any parabolic presentation.
\end{Theorem C}

Surprisingly, the truncation relations are significantly simpler when, either: the Lie algebra is not of type {\sf C}, or; the smallest two parts of partition of the nilpotent orbit have distinct sizes. We note that the the Corollary to Theorem A is significantly stronger than the result obtained in \cite[Theorem~1.3]{To23} thanks to these simplified truncation relations. This simplification depends crucially on the parabolic presentations of Theorem~B.

One technical detail which we have suppressed in this introduction is the notion of an admissible shape for $\sigma$, introduced by Brundan--Kleshcev in order to define parabolic presentations. In Section~\ref{ss:admissibleshapes} we recall this definition and introduce the new notion of a $(\beta, \sigma)$-admissible shape, where $\beta = \pm 1$. The parabolic presentations in Theorem~B apply to the case $\beta = 1$.

We conjecture that a version of Theorems~A and B should hold true for $\beta = -1$, when $y_n^+(\sigma)$ is replaced by the as-yet-undefined shifted twisted Yangian for $\sp_n$ (Satake diagram of type {\sf AII}). We also expect that the shifted twisted Yangian $\Y_n(\sigma)$ should quantize certain slices in the twisted affine Grassmanian associated to $\SO_n$, in parallel with \cite{KWWY14}.


\subsection*{Acknowledgements} Both authors would like to thank Jon Brown, Jon Brundan, Kang Lu, Hiraku Najajima, Yung-Ning Peng, Thomas Tappeiner, Weiqiang Wang and Matt Westaway for useful conversations during the development of this work. The first author is grateful to the University of Bath for funding his PhD studies. The second author would like to acknowledge funding from the UKRI Future Leaders Fellowship, grant numbers MR/S032657/1, MR/S032657/2, MR/S032657/3. 

\section{Preliminaries}

In this section we record the combinatorics, as well as some preliminary definitions and theorems to be used in the sequel. All vector spaces, algebras and algebraic varieties are defined over $\C$.

\subsection{Symmetric shift matrices and partitions}
\label{ss:symmetricshift}

Fix $n > 0$. A shift matrix is an $n \times n$ matrix $\sigma = (s_{i,j})$ with entries in $\Z_{\ge 0}$ satisfying the condition
\begin{eqnarray}
\label{e:shiftmatrixdefn}
s_{i, k} = s_{i,j} + s_{j, k}
\end{eqnarray}
whenever $|i - k| = |i-j| + |j-k|$ and $1\le i,j,k \le n$. We observe that every such shift matrix has zeroes on the diagonal, and is determined entirely by the sub- and super-diagonal entries, which can be chosen arbitrarily. We will always assume that our shift matrix is {\it symmetric} meaning $s_{i,j} = s_{j,i}$ for all $i,j$. Such a matrix is determined by $(s_{i,i+1} \mid i=1,...,n-1)$. 

Now fix an integer $\ell > 0$, which we refer to as {\it the level}. If we set
\begin{eqnarray}
\label{e:bigNdefn}
N = n\ell + 2(n-1) s_{1,2} + 2(n-2) s_{2,3} + \cdots +2 s_{n-1,n}
\end{eqnarray}
then the data $(\sigma, \ell)$ give rise to a partition $\lambda = (\lambda_1 \le \lambda_2 \le \cdots \le \lambda_n) \vdash N$ by
\begin{eqnarray}
\label{e:lambdafromsigmaell}
\begin{array}{l}
\lambda_1 = \ell,\\
\lambda_i = \lambda_{i-1} + 2s_{i-1,i} \text{ \  for \ } i=2,...,n.
\end{array}
\end{eqnarray}
We remark that all parts of $\lambda$ have the same parity.

The following easy lemma will be used often later on.
\begin{Lemma}
\label{L:lambdabij}
The assignment $(\sigma, \ell) \mapsto \lambda$ determined by \eqref{e:lambdafromsigmaell} defines a bijection from symmetric shift matrices with a fixed level, to partitions of positive integers with precisely $n$ parts, all of the same parity.$\hfill\qed$
\end{Lemma}

\subsection{Admissible shapes}
\label{ss:admissibleshapes}
Keep a choice of $n\times n$ shift matrix $\sigma$. A {\it shape of size $n$} is a tuple $\nu = (\nu_1, \nu_2,...,\nu_m)$ of positive integers such that $\sum_i \nu_i = n$, and we write $\bnu = (\bnu_1,...,\bnu_m)$ where $\bnu_a = \Sigma_{b=1}^a \nu_b$. We refer to $m$ as the {\it length} of $\nu$.

We say that $\nu$ is a {\it $\sigma$-admissible shape} if $s_{i,j} = 0$ for all $\bnu_a + 1 \le i, j < \bnu_{a+1}$ and $a=1,...,m-1$. In \cite{BK06} these were simply referred to as admissible shapes. We note that $(1^n)$ is $\sigma$-admissible for every $\sigma$, and that for each $\sigma$ there is a unique $\sigma$-admissible shape of minimal length, which is known as the {\it minimal admissible shape}.

From a $\sigma$-admissible shape we can construct a {\it relative shift matrix} $\sigma(\nu)$, which is an $m \times m$-matrix with entries $s(\nu)_{i,j}$ in $\Z_{\ge 0}$ determined by
\begin{eqnarray}
s(\nu)_{a,b} = s_{\bnu_a, \bnu_b} \text{ \ for all \ } 1\le a,b\le m.
\end{eqnarray}
Since $s_{i, i+1} = 0$ for all $i = \bnu_a, \bnu_a+1,...,\bnu_{a+1}-1$, and $\sigma$ can be recovered from its super-diagonal entries, it follows that $\sigma$ can be recovered from the pair $(\sigma(\nu), \nu)$.

Similarly, if we fix a level $\ell$ and let $\lambda = (\lambda_1,...,\lambda_n)$ be the partition determined by the data $(\sigma, \ell)$ then we can define the {\it relative partition} $\lambda(\nu) = (\lambda(\nu)_1,...,\lambda(\nu)_m)$ by $\lambda(\nu)_a = \lambda_{\bnu_a}$. As above, we may recover $\lambda$ from the pair $(\lambda(\nu), \nu)$.

Now fix $\beta = \pm 1$. We say that a $\sigma$-admissible shape $\nu$ is {\it $(\beta , \sigma)$-admissible} if $\beta^{\nu_a} = \beta$ for $a=1,...,m$. When $\beta =-1$ this ensures that $\nu_a$ are all even, whilst for $\beta=1$ any $\sigma$-admissible shape is $(\beta, \sigma)$-admissible.

The notation of a $(\beta, \sigma)$-admissible shape is key to our later constructions and so we would like to motivate it here. In Theorem~\ref{dirredcurrent} we introduce an involution on the shifted current Lie algebra $\cc_n(\sigma)$, and the involution extends to the current algebra $\cc_n$. On the zero graded piece $\gl_n \subseteq \cc_n$, the matrix of the involution is symmetric when $\beta = 1$ and antisymmetric when $\beta = -1$. This is a shadow of the fact, central to our work: the (semiclassical) shifted twisted Yangian of type {\sf AI} surjects onto classical finite $W$-algebras for which there are no multiplicity constraints on the partition of the nilpotent element, whilst the (as-yet-undefined) shifted twisted Yangian of type {\sf AII} surjects onto those algebras for which there are multiplicity constraints (Cf. Section~\ref{ss:nilpotentorbits}). In this paper the role of the shifted twisted Yangian of type {\sf AII} is played by the Dirac reductions with $\beta = -1$ appearing in Theorem~C.

\subsection{Classical finite $W$-algebras}
\label{ss:classicalFWA}

Let $\g$ be the Lie algebra of a connected complex reductive group $G$. A nilpotent orbit of $\O \subseteq \g$ is a $G$-orbit consisting of elements which act nilpotently on every finite dimension representation. Let $\kappa$ be a choice of $G$-invariant non-degenerate bilinear form on $\g$.

The Jacobson--Morozov theorem states that nilpotent $G$-orbits are in 1-1 correspondence with the $G$-orbits in the space of Lie algebra homomorphisms $\Hom(\sl_2, \g)$. The bijection is provided by choosing a non-zero nilpotent element $e\in \sl_2$ and sending $\phi\in\Hom(\sl_2, \g)$ to $G\cdot \phi(e)$. Inverting this map we may associate an $\sl_2$-triple $(e,h,f)$ to any nilpotent element $e\in \g$, unique up to conjugacy.

The grading $\g = \bigoplus_{i\in \Z} \g(i)$ induced by $\ad(h)$ is called the Dynkin grading, and the variety $e + \g^f \subseteq \g$ is called the Slodowy slice to $\O$. By $\sl_2$-theory $T_e \O = \ad(\g)e$ and $T_e(e+ \g^f) = \g^f$ are mutually orthogonal complements inside $\g$ with respect to the Killing form, and so the slice is transversal to the orbit of $e$ (in fact, to every orbit which it intersects).


Let $G(<0)$ be the unipotent subgroup of $G$ with Lie algebra $\g(<0) = \bigoplus_{i\le 0} \g(i)$. Thanks to \cite[Lemma~2.1]{GG02} we have an isomorphism
\begin{eqnarray}
\label{e:GGlem}
G(<0) \times (e+\g^f) \to e + \g(\le 1)
\end{eqnarray}
of $G(<0)$-varieties, where $G(<0)$ acts by left multiplication on the left hand factor, and via the adjoint action on the right hand side.
We define the {\it classical finite $W$-algebra} to be the subquotient
\begin{eqnarray}
S(\g,e) := (S(\g) / S(\g) \g(\le -1)_e)^{G(<0)}
\end{eqnarray}
where $\g(\le 0)_e := \{x - \kappa(e,x) \mid x\in \g(\le 0)\}$. Using \eqref{e:GGlem} we identify $S(\g,e)$ with $\C[e + \g^f]$ as commutative algebras.

We can equip $S(\g,e)$ with a Poisson structure, using a version of Hamiltonian reduction (see \cite[\textsection 3.1]{ACET20}, for example). Describing this Poisson structure is the main goal of the present article.

\subsection{Nilpotent orbits in classical Lie algebras}
\label{ss:nilpotentorbits}

Pick $\ve = \pm 1$ and an integer $N > 0$ satisfying $\ve^N = 1$. Suppose that $(\cdot, \cdot) : \C^N \times \C^N \to \C$ is a form satisfying $(u,v) = \ve (v,u)$ for all $u,v\in \C^N$. The Lie algebra subalgebra $\g_0$ of $\gl_N$ consisting of skew-self adjoint matrices with respect to $(\cdot, \cdot)$ is a classical simple Lie algebra. More precisely $\g_0 \cong \sp_N$ when $\ve= -1$ and $\g_0 \cong \so_N$ when $\ve = 1$. The connected component of the subgroup of $\GL_N$ preserving the form $(\cdot, \cdot)$ is denoted $G_0$.

The nilpotent $\GL_N$-orbits in $\gl_N$ are parameterised by partitions $\lambda \vdash N$, thanks to the theory of Jordan normal form. If $\O_\lambda$ is the orbit corresponding to $\lambda \vdash N$ then $\O_\lambda \cap \g_0$ is either empty or is a union of at most two $G_0$-orbits \cite[\textsection 5]{CM93}. In particular, $\O_\lambda \cap \g_0 \ne \emptyset$ if and only if:
\begin{itemize}
\item $(\ve = 1)$ all even parts of $\lambda$ occur with even multiplicity,
\item $(\ve = -1)$ all odd parts of $\lambda$ occur with even multiplicity.
\end{itemize}
We denote the set of partitions satisfying these conditions by $\P_\ve(N)$. A partition $\lambda \in \P_\ve(N)$ is called {\it even} if all parts of $\lambda$ have the same parity.

Also  for $\lambda \in \P_\ve(N)$ we note that $\O_\lambda \cap \g_0$ is a union of two $G_0$-orbits if and only if $\ve = 1$ and all parts of $\lambda$ are even; such partitions are called {\it very even}. In all other cases, $\O_\lambda \cap \g_0$ is a single orbit.

\subsection{Dirac reduction for fixed points}
\label{ss:Diracforfixed}

Let $A$ be an affine Poisson algebra with a group $H$ of Poisson automorphisms. The ideal of coinvariants is denoted $I_H = (h \cdot a - a\mid h\in H, \ a\in A)$. The $H$-fixed points $I_H^H \subseteq I_H$ is a Poisson ideal of $A^H$ and the Dirac reduction of $A$ by $H$ is the Poisson algebra
$$R(A, H) := A^H /I_H^H.$$
It is explained in \cite[\textsection 2.1]{To23} that when $X = \Spec(A)$ is an affine Poisson variety the Dirac reduction equips $X^H$ with the structure of an affine Poisson scheme, see also \cite[\textsection 5.4.3]{LPV13}.

Now retain the notation of Section~\ref{ss:classicalFWA}. We suppose that $H \subseteq \Aut(\g)$ is a reductive group of automorphisms fixing the $\sl_2$-triple pointwise. Then $H$ induces a group of Poisson automorphisms of $S(\g,e)$. We write $\g_0 = \g^H$ for the fixed point subalgebra, which is also the Lie algebra of a reductive group \cite[Lemma~2.5(2)]{To23}. One of the main applications of Dirac reduction, in the context of the present work, is the following.

\begin{Theorem}
\cite[Theorem~2.6]{To23}
There is an isomorphisms of Poisson algebras
\begin{eqnarray*}
R(S(\g,e), H) \isoto S(\g_0, e).
\end{eqnarray*}
\end{Theorem}

\section{Shifted twisted Yangian versus Dirac reduction of shifted Yangians}

\subsection{Twisted Yangian}

We refer the reader to \cite[Ch.~2]{Mo07} for a good introduction to Olshanski's theory of twisted Yangians.

Recall that the twisted Yangian $\Y_n$ associated to $\mathfrak{o}_n$ is generated by elements $\{ s_{i,j}^{(r)} \mid 1\le i,j \le n, \ \ r \ge 0\} $ subject to the quaternary and symmetry relation \cite[Proposition 2.2.1]{Mo07}. In \cite{LWZ23} those authors solved the long-standing problem of providing a Drinfeld presentation of $\Y_n$ associated to the split bilinear form on $\C^n$. These are known as twisted Yangians of type {\sf AI}, and from now on we will focus on type {\sf AI} type.

We recall the main results of their their paper, which is the Drinfeld presentation of $\Y_n$, see \cite[Theorem~5.1]{LWZ23}.
\begin{Theorem}
\label{T: Twisted Drinfeld}
Let $C=(c_{ij})_{1\le i,j\le n-1}$ be the Cartan matrix of type ${\sf A}_{n-1}$, and set $h_{i,-1}=1$, $c_{0j}=-\delta_{j,1}$. Then the twisted Yangian $\Y_n$ is generated by
\begin{eqnarray}
\label{e: Twisted generators}
\{h_{i,r},  b_{j,s} \mid 0 < i<N \ \ 1\le j<N, \ \ r,s \ge 0\}
\end{eqnarray}
 subject only to the following relations, for $r,s \in \mathbb{N}$:
\begin{align}
[ h_{i,r}, h_{j,s}]& =0, \label{hhN} \ \ \  \ \ { h_{i,2r}=0, } \\	
[ h_{i,r+1}, b_{j,s}]-[ h_{i,r-1}, b_{j,s+2}]&=c_{ij}[h_{i,r-1}, b_{j,s+1}]_- +\frac{1}{4}c_{ij}^2[ h_{i,r-1}, b_{j,s}],\label{hbN}\\
[ b_{i,r+1}, b_{j,s}]-[ b_{i,r}, b_{j,s+1}]&=\frac{c_{ij}}{2} [ b_{i,r}, b_{j,s}]_- -2\delta_{ij}(-1)^r h_{i,r+s+1},\label{bbN}\\
[b_{i,r},b_{j,s}]&=0,\qquad  \text{ for }|i-j|>1,\label{bbN2}\\
\mathrm{Sym}_{k_1,k_2}\big[b_{i,k_1},[b_{i,k_2},b_{j,r}] \big] &=(-1)^{k_1}[h_{i,k_1+k_2+1},b_{j,r-1}],
\quad \text{ for } c_{ij}=-1.\label{serreN}
\end{align}
where $\mathrm{Sym}_{k_1, k_2}$ denotes the sum of the two expressions obtained by permuting $k_1, k_2$ and $[ a, b]_- := ab + ba$  denotes the anticommutator.
\end{Theorem}

We equip $\Y_n$ with the {\it loop filtration}, by placing $h_{i,r}$ and $b_{i,r}$ in degree $r$. We also equip $\Y_n$ with the {\it canonical filtration} by placing $h_{i,r}$ and $b_{i,r}$ in degree $r + 1$.  We will denote the corresponding associated graded algebras by $\gr_L \mathcal{Y}_{n}$ and $\gr_C \mathcal{Y}_{n}$ for the loop and canonical filtrations, respectively.

Analogously to \cite[Proof of Theorem 5.1]{LWZ23} we inductively define elements
\begin{eqnarray}
\label{e:extra gens}
\begin{array}{ll}
b_{i+1,i}^{(r)} := b_{i,r} \in \Y_n \ \ \text{ for } \ \ 1< i \le n,\\
b _{i+1, j}^{(r)} = [b_{i+1, i}^{(0}, b_{i,j}^{(r)} ] \in \Y_n \ \ \text{ for } \ \  1\le j < i \le n \\
b_{i,i}^{(r)}:=h_{0,r}+\dots+h_{i-1,r}
\end{array}
\end{eqnarray}
and let $b_{i,j}^{(r)}:=(-1)^{(r-1)}b_{j,i}^{(r)}$ for $i < j$.
\subsection{Shifted twisted Yangian and the semiclassical limit}
\label{ss:shiftedtwisted}

Retain the notation of Theorem~\ref{T: Twisted Drinfeld}. Let $\sigma = (s_{i,j})_{1,\le i, j \le N}$ be a symmetric $n \times n$ shift matrix. We define the shifted twisted Yangian $\Y_n(\sigma)$ to be the associative algebra generated by the elements
\begin{eqnarray}
\label{e:twisted shifted generators}
\{h_{i,r},  b_{j,s} \mid 0 < i<N \ \ 1\le j<N, \ \ r \ge 0, s \ge s_{j, j+1} \}
\end{eqnarray}
subject to relations \eqref{hhN}--\eqref{serreN} for all admissible indexes (i.e. whenever the indexes make sense). This algebra is equipped with a loop filtration, again placing $h_{i,r}$ and $b_{i,r}$ in degree $r$, and a canonical filtration which places those same elements in degree $r+1$.
We will denote the corresponding associated graded algebras by $\gr_L \mathcal{Y}_{n}(\sigma)$ for the loop filtration and by $\gr_C \mathcal{Y}_{n}(\sigma)$ for the canonical filtration respectively. 
Similar to \eqref{e:extra gens} we define elements
\begin{eqnarray}
\label{e:twisted shifted extra gens}
\begin{array}{ll}
{}^\sigma b_{i+1,i}^{(r)} := b_{i,r} \in \Y_n(\sigma) \ \ \text{ for } \ \ 1< i \le n, \ r \geq s_{i+1,i}\\
{}^\sigma b_{i+1, j}^{(r)}= [{}^\sigma b_{i+1, i}^{(s_{i,i+1})}, {}^\sigma b_{i,j}^{(r-s_{i+1,i})}] \in \Y_n(\sigma) \ \ \text{ for } \ \  1\le j < i \le n, \ r \geq s_{i+1,j} \\
{}^{\sigma}b_{i,i}^{(r)}:=h_{0,r}+\dots+h_{i-1,r}
\end{array}
\end{eqnarray}
and let ${}^{\sigma}b_{i,j}^{(r)}:=(-1)^{(r-1)}{}^{\sigma}b_{j,i}^{(r)}$ for $i < j$.

\begin{Theorem}\label{TwistedshiftedPBW}
\begin{enumerate}
\label{T:Gens to Gens}
\setlength{\itemsep}{4pt}
\item The map $\Y_n(\sigma) \to \Y_n$ sending 
\begin{eqnarray}
\begin{array}{rcl}
b_{i, r} & \longmapsto &  b_{i,r}\\
h_{i,r} & \longmapsto & h_{i,r}.
\end{array}
\end{eqnarray}
is an injective filtered homomorphism with respect to both the loop filtration and the canonical filtration.
\item $\Y_n(\sigma)$ has a PBW basis consisting of ordered monomials in the elements $$\{ h_{i,r}, {}^\sigma b_{j,k}^{(s)} \mid 0 \le i < n, \ \ 1\le k < j  \le n,\ \ r \ge 0, \ \  s \ge s_{j, k} \}.$$
\end{enumerate} 
\end{Theorem}
\begin{proof}
It is clear that the map is a filtered homomorphism with respect to both filtrations. As the  map $\mathcal{Y}_N(\sigma) \to \mathcal{Y}_N$ is filtered with respect to the loop filtration it induces a map $\gr_L \mathcal{Y}_N(\sigma) \to \gr_L \mathcal{Y}_N$. We will write $\bar{h}_{i,2r+1}, \ {}^{\sigma}\bar{b}_{i,j}^{(r)}$ for the images of $h_{i,2r+1}, \ {}^{\sigma}b_{i,j}^{(r)}$ in $\gr_L \mathcal{Y}_N(\sigma)$. It then follows from \cite[(5.3)]{LWZ23} that the map on the associated graded sends ${}^{\sigma}\bar{b}_{i,j}^{(r)} \in \gr_L \mathcal{Y}_N(\sigma)$ to $\bar{b}_{i,j}^{(r)} \in \gr_L \mathcal{Y}_N$ as defined in \cite[§5.4]{LWZ23} but from the proof of \cite[Theorem 5.1]{LWZ23} we see that the ordered monomials in the elements $\bar{h}_{i,2r+1}, \ \bar{b}_{i,j}^{(r)}$ are linearly independent in $\gr_L \mathcal{Y}_n$. In particular the ordered monomials in the elements  $h_{i,2r+1}, {}^{\sigma}b_{i,j}^{(r)} \in \mathcal{Y}_N(\sigma) $ are linearly independent too.

Following the proof of \cite[Theorem 5.1]{LWZ23} verbatim it is easy to see that the ordered monomials in the elements $h_{i,2r+1}, {}^{\sigma}b_{j,k}^{(s)}$ for $0 \leq i < N, \ 1 \leq k<j \leq N, \ r \in \mathbb{N}, \ s \geq s_{j,k}$ span $\mathcal{Y}_n(\sigma)$. Hence the map is injective.
\end{proof}
Recall the definition of the shifted current algebra $\cc_n(\sigma) \subseteq \gl_n[z]$ from \cite[(3.3)]{To23}. Let $\tau:\gl_n \to \gl_n: A \mapsto -A^T$. We define the \textit{twisted shifted current algebra} for $\mathfrak{so}_n$  to be $\cc_n(\sigma)^{\tau}:=\{A(z) \in \cc_n(\sigma) \mid A(z)= \tau(A(-z)) \}$.
\begin{Lemma}\label{TwistedYPBWS}
\begin{enumerate}
\item We have a natural isomorphism $U(\cc_n(\sigma)^{\tau}) \to \gr_L \Y_n(\sigma)$ given by
\begin{eqnarray}
e_{i,j}z^r+(-1)^{r-1}e_{j,i}z^r \mapsto {}^{\sigma}b_{i,j}^{(r)}+F^{L}_{r-1}\Y_n
\end{eqnarray}
\item We have a natural isomorphism of commutative algebras $S(\cc_n(\sigma)^{\tau}) \to \gr_C \Y_n(\sigma)$ given by
\begin{eqnarray}
e_{i,j}z^r+(-1)^{r-1}e_{j,i}z^r \mapsto {}^{\sigma}b_{i,j}^{(r)}+F^{C}_{r}\Y_n
\end{eqnarray}
\end{enumerate}
\end{Lemma}
\begin{proof}
We let ${}^{\sigma}\bar{b}_{i,j}^{(r)}$ denote ${}^{\sigma}{b}_{i,j}^{(r)}+F^{L}_{r-1}\Y_n(\sigma)$.
Following the proof of  \cite[(5.11)]{LWZ23} verbatim we obtain the relation
\begin{equation}
[{}^{\sigma}\bar{b}_{i,j}^{(r)},{}^{\sigma}\bar{b}_{k,l}^{(s)}]=\delta_{j,k}{}^{\sigma}\bar{b}_{i,l}^{(r+s)}-\delta_{l,i} {}^{\sigma}\bar{b}_{k,j}^{(r+s)}-(-1)^r\delta_{k,i}{}^{\sigma}\bar{b}_{j,l}^{(r+s)}+\delta_{l,j}(-1)^{(r)}{}^{\sigma}\bar{b}_{k,i}^{(r+s)}
\end{equation}
and hence the map in (1) is a homomorphism of algebras. It is easy to see that $\cc_n(\sigma)^{\tau}$ has a basis consisting of $e_{i,j}z^r+(-1)^{r-1}e_{j,i}z^r \, i>j,r \geq s_{i,j}$  together with $e_{i,i}z^{2r+1}$ and hence the ordered monomials in these elements form a basis for $U(\gl_n[z]^{\tau})$. By Theorem \ref{TwistedshiftedPBW} the ordered monomials in the elements ${}^{\sigma}b_{i,j}^{(r)} \, i>j, \ r \geq s_{i,j}$  together with $b_{i,i}^{(2r+1)}$ form a basis for $\Y_n(\sigma)$ and hence our map is an isomorphism.
For part  (2) using the definition of  the twisted Yangian as in \cite[Definition 2.1]{LWZ23} we define a filtration by putting the generators $s_{i,j}^{(r)}$ in degree $r$.
By \cite[Theorem 2.4.3]{Mo07} we see that the associated graded algebra is commutative with bracket in degree $-1$. Using  \cite[(3.1), (3.12)-(3.13)]{LWZ23} it follows that this filtration lines up with the canonical filtration on $\Y_n$ and hence  $\gr_C \Y_n$ is commutative. The map defined in Theorem \ref{TwistedshiftedPBW} is a filtered injection and hence also $\gr_C \Y_n(\sigma)$ is commutative.
 The current theorem then follows by part (1).
\end{proof}

Before we proceed we take a brief digression to discuss presentations of Poisson algebras. The free Poisson algebra generated by a set $X$ is the symmetric algebra $S(L_X)$ of the free Lie algebra $L_X$ generated by the set $X$. A {\it Poisson presentation} of a Poisson algebra $A$ is an exact sequence of the form
$$I \to S(L_X) \to A \to 0$$
where $I \subseteq S(L_X)$ is a Poisson ideal. The generators of the Poisson ideal will be be referred to as {\it Poisson relations}. In the sequel it will be important to distinguish between between generators of $A$ viewed as a commutative algebra and as a Poisson algebra, and so we emphasise this terminology whenever any ambiguity is possible. We refer the reader to \cite[\textsection 3.1]{To23} for slightly more detail.

\begin{Proposition}
\label{P:someprop}
Let $A=(c_{ij})_{1\leq i,j\leq n-1}$ be the Cartan matrix of type ${\sf A}_{n-1}$, and set $h_{i,-1}=1$, $c_{0j}=-\delta_{j,1}$. Then $\gr_C \mathcal{Y}_n(\sigma)$ is Poisson generated by $h_{i,r}$, $ b_{j,s}$, for $0\leq i<n$, $1\leq j<n$ and $r\in\mathbb{N}, s \geq s_{i+1,i}$ subject only to the following relations:
\begin{align}
\{ h_{i,r}, h_{j,s} \}&=0, \qquad  h_{i,2r}=0,
\label{hhN2}\\	
\{ h_{i,r+1}, b_{j,s}\}-\{ h_{i,r-1}, b_{j,s+2} \}&=2 c_{ij}h_{i,r-1}b_{j,s+1},\label{hbN2}\\
 \{b_{i,r+1}, b_{j,s}\}-\{b_{i,r}, b_{j,s+1}\}&=c_{ij}b_{i,r}b_{j,s}-2\delta_{ij}(-1)^r h_{i,r+s+1},\label{bbN21}\\
\{b_{i,r},b_{j,s}\}&=0,\qquad  \text{ for }|i-j|>1,\label{bbN22}\\
\mathrm{Sym}_{k_1,k_2}\{b_{i,k_1},\{b_{i,k_2},b_{j,r}\} \} &=(-1)^{k_1}\{h_{i,k_1+k_2+1},b_{j,r-1}\},
\quad \text{ for } c_{ij}=-1.\label{serreN2}
\end{align}
\end{Proposition}
\begin{proof}
Let $\widehat{y}_n(\sigma)$ be the algebra with generators and relations as in the statement of the Proposition. From relations \eqref{hhN}--\eqref{serreN} is follows that there is a canonical surjection $\widehat{y}_n(\sigma)$. We inductively define elements in  $\widehat{y}_n(\sigma)$
\begin{eqnarray}
\begin{array}{ll}
{}^\sigma b_{i+1,i}^{(r)} := b_{i,r} \in \widehat{y}_n(\sigma) \ \ \text{ for } \ \ 1< i \le n, \ r \geq s_{i+1,i}\\
{}^\sigma b_{i+1, j}^{(r)}= \{{}^\sigma b_{i+1, i}^{(s_{i,i+1})}, {}^\sigma b_{i,j}^{(r-s_{i+1,i})}\} \in \widehat{y}_n(\sigma) \ \ \text{ for } \ \  1\le j < i \le n, \ r \geq s_{i+1,j} \\
{}^{\sigma}b_{i,i}^{(r)}:=h_{0,r}+\dots+h_{i-1,r} \in \widehat{y}_n(\sigma) \ \ \text{ for } \ \ 1\le i \le n, \ r \geq 0
\end{array}
\end{eqnarray}
and let ${}^{\sigma}b_{i,j}^{(r)}:=(-1)^{(r-1)}{}^{\sigma}b_{j,i}^{(r)}$ for $i < j$. We also define a loop filtration on $\widehat{y}_n(\sigma)$ by putting the generators in degree $r$ and denote the counterparts of the $^\sigma b_{i,j}^{(r)}$ in $\gr_L\widehat{y}_n(\sigma)$ by $^\sigma \bar{b}_{i,j}^{(r)}$.
 Once more following the proof of \cite[(5.11)]{LWZ23} verbatim we get 
 \begin{equation}
\{{}^{\sigma}\bar{b}_{i,j}^{(r)},{}^{\sigma}\bar{b}_{k,l}^{(s)}\}=\delta_{j,k}{}^{\sigma}\bar{b}_{i,l}^{(r+s)}-\delta_{l,i} {}^{\sigma}\bar{b}_{k,j}^{(r+s)}-(-1)^r\delta_{k,i}{}^{\sigma}\bar{b}_{j,l}^{(r+s)}+\delta_{l,j}(-1)^{(r)}{}^{\sigma}\bar{b}_{k,i}^{(r+s)}
\end{equation}
so that $\widehat{y}_n(\sigma)$ is spanned by monomials in the elements $h_{i,2r+1},{}^{\sigma}b_{i,j}^{(r)} \in \widehat{y}_n(\sigma)$. By Lemma \ref{TwistedYPBWS} $\gr_C \mathcal{Y}_{n}(\sigma)$ is a polynomial ring in the images of these maps under the canonical surjection so that the canonical surjection is an isomorphism. 
\end{proof}

\subsection{The Dirac reduction of a shifted twisted Yangian}

We continue to fix $n > 0$, choose an $n\times n$ symmetric shift matrix $\sigma$ and fix $\ve \in \{\pm 1\}$. The {\em (semiclassical) shifted Yangian} $y_n(\sigma)$ is the Poisson algebra with Poisson generators
\begin{equation}\label{e:Ybasicgens}
\begin{array}{c} \{d_i^{(r)} \mid 1\leq i \leq n, \ 0 < r\} \cup \{e_{i}^{(r)} \mid 1\leq i < n, \ s_{i,i+1} < r \} \\ \cup \, \{f_{i}^{(r)} \mid 1\leq i < n, \ s_{i+1,i} < r\} \end{array}
\end{equation}
and Poisson relations \cite[(3.37)-(3.48)]{To23}. Following \cite[(3.58)]{To23} we introduce an involution on $y_n(\sigma)$
\begin{eqnarray}
\label{e:tauon1}
\begin{array}{rcl}
\tau(d_i^{(r)})& := & (-1)^{r} d_{i}^{(r)}; \\ [3pt]
\tau(e_{i}^{(r)}) & := & (-1)^{r + s_{i,i+1}} f_i^{(r)}; \\ [3pt]
\tau(f_i^{(r)}) & := & (-1)^{r + s_{i+1, i}} e_i^{(r)} .
\end{array}
\end{eqnarray}

\begin{Theorem}
\label{T:presentDirac}
Let $\sigma$ be a symmetric shift matrix, let $y_n(\sigma)$ denote the corresponding semiclassical shifted Yangian and let $\tau$ denote the automorphism of $y_n(\sigma)$ introduced in \eqref{e:tauon1}. The Dirac reduction $R(y_n(\sigma), \tau)$ is Poisson generated by elements
\begin{eqnarray}
\label{e:DYgens}
\{\eta_i^{(2r)} \mid 1\le i \le n, \ 0 < r \} \cup \{ \theta_i^{(r)} \mid 1\le i < n, \ s_{i,i+1} < r \}
\end{eqnarray}
together with the following relations
\begin{eqnarray}
\label{e:dyrel1}
& & \{\eta_i^{(2r)}, \eta_j^{(2s)}\} = 0 \\ 
\label{e:dyrel2}
& &\big\{\eta_i^{(2r)}, \theta_j^{(s)}\big\} = (\delta_{i,j} - \delta_{i,j+1}) \sum_{t=0}^{r-1} \eta_i^{(2t)} \theta_j^{(2r+s -1- 2t)}
\end{eqnarray}
\begin{eqnarray}
\label{e:dyrel34}
\{\theta_i^{(r)}, \theta_i^{(s)}\} = \frac{1}{2} \sum_{t=r}^{s-1} \theta_i^{(t)} \theta_i^{(r+s-1-t)} + (-1)^{s_{i,i+1}}\varpi_{s,r} \sum_{t=0}^{(r+s-1)/2} \eta_{i+1}^{(r+s-1-2t)} \teta_{i}^{(2t)} & \text{ for } & r < s.\\
\label{e:dyrel5}
\big\{\theta_i^{(r+1)}, \theta_{i+1}^{(s)}\big\} - \big\{\theta_i^{(r)}, \theta_{i+1}^{(s+1)}\big\} = \frac{1}{2}\theta_i^{(r)} \theta_{i+1}^{(s)} & &
\end{eqnarray}
\begin{eqnarray}
\label{e:dyrel6}
\big\{\theta_i^{(r)}, \theta_j^{(s)}\big\} = 0 & & \text{ for } |i - j| > 1\\
\label{e:dyrel7}
\Big\{\theta_i^{(r)}, \big\{\theta_i^{(s)}, \theta_j^{(t)}\big\}\Big\} + \Big\{\theta_i^{(s)}, \big\{\theta_i^{(r)}, \theta_j^{(t)}\big\}\Big\} = 0 & & \text{ for } |i - j| =1 \text{ and } r + s \text{ odd}
\end{eqnarray}
\begin{eqnarray}
\label{e:dyrel8}
& &\Big\{\theta_i^{(r)}, \big\{\theta_i^{(s)}, \theta_j^{(t)}\big\}\Big\} + \Big\{\theta_i^{(s)}, \big\{\theta_i^{(r)}, \theta_j^{(t)}\big\}\Big\} =  \\
& & \nonumber  \ \ \ \ \ \ \ \ \ \  2(-1)^{s+s_{i,i+1}-1} (\delta_{i,j+1}+\delta_{i+1,j}) \sum_{m_1=0}^{m-1} \sum_{m_2=0}^{m-m_1-1} \eta_{i+1}^{(2m_2)}  \teta_i^{(2m_1)} \theta_j^{(2(m_1-m_2) + t)} \\
& & \nonumber \ \ \ \ \ \ \ \ \ \ \ \ \ \ \ \ \ \ \ \ \ \ \ \ \ \ \ \ \ \ \text{ for } |i - j| =1 \text{ and } r + s = 2m \text{ even}
\end{eqnarray}
where we adopt the convention $\eta_i^{(0)} = \teta_i^{(0)} = 1$ and the elements $\teta_i^{(2r)} $ are defined via the recursion
\begin{eqnarray}
\label{e:tetatwisteddefinition}
\teta_i^{(2r)} := -\sum_{t=1}^r \eta_i^{(2t)} \teta_i^{(2r-2t)}.
\end{eqnarray}
\end{Theorem}
\begin{proof}
\eqref{e:dyrel1}-\eqref{e:dyrel7} are as in \cite[Theorem~3.7]{To23}. We claim that \eqref{e:dyrel8} is equivalent to \cite[(3.74)]{To23} which says 
\begin{eqnarray*}
& &\Big\{\theta_i^{(r)}, \big\{\theta_i^{(s)}, \theta_j^{(t)}\big\}\Big\} + \Big\{\theta_i^{(s)}, \big\{\theta_i^{(r)}, \theta_j^{(t)}\big\}\Big\} =  \\
& & \nonumber  \ \ \ \ \ \ \ \ \ \  2(-1)^{s+s_{i,i+1}-1} \delta_{i,j+1} \sum_{m_1=0}^{m-1} \sum_{m_2=0}^{m_1} \eta_{i+1}^{(2(m-m_1-1))}  \teta_i^{(2m_2)} \theta_j^{(2(m_1-m_2) + t)} \\
& & \nonumber \ \ \ \ \ \ \ \ \ \ \ \ \ \ \ \ \ \ \ \ + 2(-1)^{s+s_{i,i+1}-1} \delta_{i+1,j} \sum_{m_1=0}^{m-1} \sum_{m_2=0}^{m-m_1-1} \teta_i^{(2m_1)}\eta_{i+1}^{(2m_2)} \theta_j^{(2(m-m_1-m_2 - 1) + t)} \\
& & \nonumber \ \ \ \ \ \ \ \ \ \ \ \ \ \ \ \ \ \ \ \ \ \ \ \ \ \ \ \ \ \ \text{ for } |i - j| =1 \text{ and } r + s = 2m \text{ even}
\end{eqnarray*}
Making the substitution $q:=m-m_1-1$ in the first sum this now reads as 
\begin{eqnarray*}
& &\Big\{\theta_i^{(r)}, \big\{\theta_i^{(s)}, \theta_j^{(t)}\big\}\Big\} + \Big\{\theta_i^{(s)}, \big\{\theta_i^{(r)}, \theta_j^{(t)}\big\}\Big\} =  \\
& & \nonumber  \ \ \ \ \ \ \ \ \ \  2(-1)^{s+s_{i,i+1}-1} \delta_{i,j+1} \sum_{q=0}^{m-1} \sum_{m_2=0}^{m-q-1} \eta_{i+1}^{(2(q))}  \teta_i^{(2m_2)} \theta_j^{(2(m-q-1-m_2) + t)} \\
& & \nonumber \ \ \ \ \ \ \ \ \ \ \ \ \ \ \ \ \ \ \ \ + 2(-1)^{s+s_{i,i+1}-1} \delta_{i+1,j} \sum_{m_1=0}^{m-1} \sum_{m_2=0}^{m-m_1-1} \teta_i^{(2m_1)}\eta_{i+1}^{(2m_2)} \theta_j^{(2(m-m_1-m_2 - 1) + t)} \\
& & \nonumber \ \ \ \ \ \ \ \ \ \ \ \ \ \ \ \ \ \ \ \ \ \ \ \ \ \ \ \ \ \ \text{ for } |i - j| =1 \text{ and } r + s = 2m \text{ even}
\end{eqnarray*}
now changing the order of summation in the first sum this reads as 
\begin{eqnarray*}
& &\Big\{\theta_i^{(r)}, \big\{\theta_i^{(s)}, \theta_j^{(t)}\big\}\Big\} + \Big\{\theta_i^{(s)}, \big\{\theta_i^{(r)}, \theta_j^{(t)}\big\}\Big\} =  \\
& & \nonumber  \ \ \ \ \ \ \ \ \ \  2(-1)^{s+s_{i,i+1}-1} \delta_{i,j+1} \sum_{m_2=0}^{m-1} \sum_{q=0}^{m-m_2-1} \eta_{i+1}^{(2(q))}  \teta_i^{(2m_2)} \theta_j^{(2(m-q-1-m_2) + t)} \\
& & \nonumber \ \ \ \ \ \ \ \ \ \ \ \ \ \ \ \ \ \ \ \ + 2(-1)^{s+s_{i,i+1}-1} \delta_{i+1,j} \sum_{m_1=0}^{m-1} \sum_{m_2=0}^{m-m_1-1} \teta_i^{(2m_1)}\eta_{i+1}^{(2m_2)} \theta_j^{(2(m-m_1-m_2 - 1) + t)} \\
& & \nonumber \ \ \ \ \ \ \ \ \ \ \ \ \ \ \ \ \ \ \ \ \ \ \ \ \ \ \ \ \ \ \text{ for } |i - j| =1 \text{ and } r + s = 2m \text{ even}
\end{eqnarray*}
which is as claimed.
\end{proof}
\begin{Remark}
Relation \eqref{e:dyrel34} is actually different from \cite[(3.70)]{To23} and we have replaced $\varpi_{r,s}$ with $\varpi_{s,r}$. This is due to the fact that \cite[(3.70)]{To23} actually should read as our relation \eqref{e:dyrel34} as is evident from the proof of \cite[Thm. 3.7]{To23} but a sign mistake appears in the relation \cite[(3.70)]{To23} which is fixed by replacing  $\varpi_{r,s}$ with $\varpi_{s,r}$.
\end{Remark}

\subsection{Proof of Theorem A}
\label{ss:Proof of A}
We now introduce certain elements in $R(y_n(\sigma),\tau)$ whose role will become clear later on. Let
\begin{align}
\bar{\eta}^{(2r)}_1&:={\eta}^{(2r)}_1 \\
\bar{\eta}^{(2r)}_i&:= \sum_{t=0}^{r}\tilde{\eta}_{i-1}^{(2t)}\eta_{i}^{(2r-2t)} \, \textnormal{for } i >1
\end{align}
\begin{Remark}
\label{etabargen}
Define the following formal power series
 \begin{align}
 \eta_{i}(u)&:=\sum_{r \geq 0}  \eta_{i}^{2r}u^{-2r}\\
 \tilde{\eta}_{i}(u)&:=\sum_{r \geq 0}  \tilde{\eta_{i}}^{2r}u^{-2r}
 \end{align}
Note that for $i >1$ the $\bar{\eta}^{(2r)}_i$ are the coefficients of $t^{-2r}$ of the power series $ \bar{\eta}_i(u):=\tilde{\eta}_{i-1}(u) \eta_{i}(u)$. By the recursion \eqref{e:tetatwisteddefinition} we see that
$ \tilde{\eta}_{i}(u) \eta_{i}(u)=1$. In particular this gives $\eta_{i}(u)= \bar{\eta}_{1}(u) \dots \bar{\eta}_{i}(u) $ so that we can recover monomials in the elements $\eta_{i}^{(2r)}, 1 \leq i \leq n, \ r \geq 0$ from monomials in the elements $\bar{\eta}_{i}^{(2r)}, 1 \leq i \leq n, \ r \geq 0$ and vice versa. In particular, the elements
\begin{eqnarray}
\label{e:DYgens}
\{\bar\eta_i^{(2r)} \mid 1\le i \le n, \ 0 < r \} \cup \{ \theta_i^{(r)} \mid 1\le i < n, \ s_{i,i+1} < r \}
\end{eqnarray}
are a set of Poisson generators of $R(y_n(\sigma), \tau)$, thanks to Theorem~\ref{T:presentDirac}.
\end{Remark}
The following Lemma will prove useful later on.
\begin{Lemma}
The following equalities hold in $R(y_n(\sigma),\tau)$
\begin{align}
\label{etabrel1}
\{\bar{\eta}^{(2r)}_i,\bar{\eta}^{(2s)}_j \}&=0 \\
\label{etabrel2}
\{ \tilde{\eta}^{(2r)}_i,\theta_j^{(s)}\}&= (\delta_{i,j+1} - \delta_{i,j}) \sum_{t=0}^{r-1} \tilde{\eta}_i^{(2t)} \theta_j^{(2r+s -1- 2t)} \\
\label{etabrel3}
\{ \bar{\eta}^{(2r)}_i,\theta_j^{(s)} \}&= (\delta_{i-1,j+1}+\delta_{i,j}-2\delta_{i,j+1})\sum_{t=0}^{r-1}\bar{\eta}_i^{(2t)}\theta_j^{(2r+s-1-2t)} \textnormal{for } i >1
\end{align}
\end{Lemma}
\begin{proof}
The first one of these relations follows immediately by relation \eqref{e:dyrel1} as $\bar{\eta}^{(2r)}_i$ is polynomial in $\eta_i^{2r},\eta_{i-1}^{2s}, \ 1 \leq r,s$.
The second relation is proved just as \cite[(3.56)]{To23}. For the third relation we have
\begin{align*}
\{ \bar{\eta}^{(2r)}_i,\theta_j^{(s)} \}&=\sum_{t=0}^{r}
\{ \tilde{\eta}_{i-1}^{(2t)}{\eta}^{(2r-2t)}_i,\theta_j^{(s)} \}\\&=(\delta_{i-1,j+1}-\delta_{i-1,j})\sum_{t=1}^{r}\eta_{i}^{(2r-2t)}\sum_{t'=0}^{t-1}\tilde{\eta}_{i-1}^{(2t')}\theta_{j}^{(2t+s-1-2t')}\\
&+(\delta_{i,j}-\delta_{i,j+1})\sum_{t=0}^{r-1}\tilde{\eta}_{i-1}^{(2t)}\sum_{t'=0}^{r-t-1}{\eta}_{i}^{(2t')}\theta_{j}^{(2r-2t+s-1-2t')}
\end{align*}
where the second equality follows by \eqref{e:dyrel2} and \eqref{etabrel2}. Thus the coefficient of $\theta_j^{(s+1+2k)} \, 0 \leq k \leq r-1$ is given by $((\delta_{i-1,j+1}-\delta_{i-1,j})+(\delta_{i,j}-\delta_{i,j+1}))\sum_{w=0}^{r-j-1}\eta_{i}^{(w)}\tilde{\eta}_{i-1}^{(2r-2w-2j-2)}$
which proves the lemma. 
\end{proof}
\begin{Theorem}\label{yangianquantizes}
The semiclassical limit of the shifted twisted Yangian $y_n^{+}(\sigma):=\gr_C \Y_n(\sigma)$  is isomorphic, as a Poisson algebra, to the Dirac reduction $R(y_n(\sigma), \tau)$, with respect to the involution \eqref{e:tauon1}.
\end{Theorem}
\begin{proof}
Consider the map $ \gr_C \Y_n(\sigma) \to R(y_n(\sigma),\tau) $ given by  
\begin{equation}
\begin{split}
h_{0,2r-1}+F_{r-1}\Y_n(\sigma) & \mapsto (2)^{2r+2}\bar{\eta}_{1}^{(2r)}   \\
h_{i-1,2r-1}+F_{r-1}\Y_n(\sigma)& \mapsto (2)^{2r+2}\bar{\eta}_{i}^{(2r)}   \textnormal{ for } i>1 \\
h_{i,2r} &\mapsto 0 \\
b_{i,r-1}+F_{r-1}\Y_n(\sigma)& \mapsto (-1)^{\frac{s_{i,i+1}  + 1}{2}}(-2)^{r}\theta_{i}^{(r)}  
\end{split}
\end{equation}
Using relations \eqref{e:dyrel2}-\eqref{e:dyrel8} together with \eqref{etabrel1} and\eqref{etabrel3} it is readily verified that the map is a morphism of Poisson algebras. Indeed to check \eqref{hhN2} use \eqref{etabrel1}, to check \eqref{hbN2} use \eqref{etabrel3}, to check \eqref{bbN21} use \eqref{e:dyrel34} and \eqref{e:dyrel5}, to check \eqref{bbN22} use \eqref{e:dyrel6} and to check \eqref{serreN2} use \eqref{e:dyrel7},\eqref{e:dyrel8} and \eqref{etabrel3}.
By the proof of \cite[Theorem 3.7]{To23} we know that $R(y_n(\sigma))$ has a basis consisting of monomials in $\eta_{i}^{(2r)},\theta_{i,j}^{(s)}, 1 \leq i < j \leq n, r \geq 0, s > s_{i,j}$ where the $\theta_{i,j}^{(s)}$ are defined inductively via the recursion
\begin{eqnarray}
\label{e:higherthetas}
\begin{array}{rcl}
\theta_{i,i+1}^{(r)} := \theta_i^{(r)} & \text{ for } & 1\le i < n, \ 0 < r \vspace{4pt} \\
\theta_{i,j}^{(r)} := \{\theta_{i,j-1}^{(r-s_{j-1, j})}, \theta_{j-1}^{(s_{j-1, j} + 1)}\} & \text{ for } & 1\le i < j\le n, \ s_{i,j} < r.\end{array}
\end{eqnarray} By definition of the map we see that ${}^{\sigma}b_{j,i}^{(r-1)}$ as defined in \eqref{e:twisted shifted extra gens}
maps to a nonzero scalar multiple of  $\theta_{i,j}^{(r)}$.

 By Remark \ref{etabargen} it follows that the set of monomials in $\bar{\eta}_{i}^{(2r)},\theta_{i,j}^{(s)}, 1 \leq i < j \leq n, r \geq 0, s > s_{i,j}$ is a basis for $R(y_n(\sigma),\tau)$ and thus our map sends a spanning set to a basis. Hence it is an isomorphism.
\end{proof}
\section{Finite $W$-algebras for classical Lie algebras}

\label{s:finiteWalgebrasforclassical}

\subsection{Centralisers in general linear Lie algebras}
\label{ss:Dynkinandcentraliser}

Fix $N > 0$ and let $\lambda$ be a partition of $N$. Following Lemma~\ref{L:lambdabij} we let $\sigma$ be the shift matrix associated to $\lambda$. Let $\nu$ be a $\sigma$-admissible shape, and resume the notation $\sigma(\nu)$ and $\lambda(\nu)$ for the relative shift matrix and relative partition, defined in Section~\ref{ss:admissibleshapes}.

We introduce symbols $\{u_{a,i,j} \mid 1\le a \le m, \ 1\le i \le \nu_m, \ 1\le j \le \lambda(\nu)_a\}$ and identify $\C^N$ with the vector space with basis $u_{a,i,j}$.  The general linear Lie algebra $\g := \gl_N$ has a basis 
\begin{equation}
\{e_{a,i,j;b,k,l} \mid 1\le a,b \le m,\ 1\le i\le \nu_a, \ 1\le k \le \nu_b,\ 1\le j\le \lambda(\nu)_a, \ 1\le l \le \lambda(\nu)_b\}
\end{equation}
where $e_{a,i,j;b,k,l} u_{c,r,s} = \delta_{b,c}\delta_{k,r} \delta_{l, s} u_{a,i,j}$, so that
\begin{eqnarray}
\begin{array}{l}
[e_{a_1,i_1,j_1;b_1,k_1,l_1}, e_{a_2,i_2,j_2;b_2,k_2,l_2}] = \vspace{4pt} \\  \, \ \ \ \ \ \ \ \ \ \ \ \ \ \ \  \ \delta_{b_1,a_2} \delta_{k_1, i_2} \delta_{l_1, j_2} e_{a_1,i_1, j_1; b_2,k_2, l_2} -  \delta_{b_2,a_1} \delta_{k_2, i_1}\delta_{l_2, j_1} e_{a_2,i_2, j_2; b_1,k_1, l_1}
\end{array}
\end{eqnarray}

We pick the nilpotent element in $\g$ by the rule
\begin{eqnarray}
\label{e:edefinition}
e :=  \sum_{a=1}^{m} \sum_{i=1}^{\nu_a} \sum_{j = 1}^{\lambda(\nu)_{a}-1} e_{a,i,j;a,i,j + 1}
\end{eqnarray}
It has Jordan blocks of sizes $\lambda_1,...,\lambda_{n}$. We also define a semisimple element $h\in \gl_N$ by
\begin{eqnarray}
\label{e:hdefinition}
h := \sum_{a=1}^{m}\sum_{i=1}^{\nu_a} \sum_{j=1}^{\lambda(\nu)_a} (\lambda(\nu)_a-1+2j)e_{a,i,j;a,i,j}
\end{eqnarray}

The pair $\{e,h\}$ can be completed to an $\sl_2$ triple $\{e,h,f\} \subseteq \g$. The Dynkin grading $\g = \bigoplus_{i\in \Z} \g(i)$ given by eigenspaces of $\ad(h)$ is a good grading, in the sense of \cite{EK05}. It is determined by
\begin{eqnarray}
\label{e:gooddegrees}
\deg(e_{a,i,j;b,k,l}) = 2(l-j) + \lambda(\nu)_a - \lambda(\nu)_b.
\end{eqnarray}
and so, in particular, $e \in \g(2)$.
 For $1\le a, b\le m$, \ $1\le i\le \nu_a$, \ $1\le k\le \nu_b$ and $r = s_{a,b}(\nu), s_{a,b}(\nu)+ 1,..., s_{a,b}(\nu) + \min(\lambda(\nu)_a, \lambda(\nu)_b) - 1$, we define elements
\begin{eqnarray}
\label{e:cijrdefinition}
c_{a,i,b,k}^{(r)} = \sum_{ 2r = 2(l-j) + \lambda(\nu)_{a}- \lambda(\nu)_{b}} e_{a,i,j;b,k,l}
\end{eqnarray}

The following description of the centraliser is given in \cite[Lemma~4.1]{To23}, see \cite[\textsection 1]{Ya10} for more detail. In \cite{To23} two indices were used to describe a basis for $\C^N$, whilst here we use three indices. Translating between these two conventions is an easy exercise.
\begin{Lemma}
\label{L:centraliserlemma}
The centraliser $\g^e$ has a basis consisting of elements
\begin{eqnarray}
\label{e:gebasis}
\{c_{a,i,b,k}^{(r)} \mid 1\le a,b \le m,  1\le i \le \nu_{a},  1\le k \le \nu_{b}, \ s_{a,b} \le r < s_{a,b} + \min(\lambda(\nu)_a, \lambda(\nu)_b)\}
\end{eqnarray}
and the Lie bracket is given by
\begin{equation}
[c_{a,i,b,j}^{(r)}, c_{c,k,d,l}^{(s)}] = \delta_{b,c} \delta_{j,k}c_{a,i,d,l}^{(r + s)} - \delta_{a,d}\delta_{i,l} c_{c,k,b,j}^{(r+s)}.
\end{equation}
Furthermore $\g^e$ is a Dynkin graded subalgebra of $\g$ with $c_{a,i,b,k}^{(r)}$ lying in degree $2r$. $\hfill \qed$
\end{Lemma}
\subsection{Symplectic and orthogonal subalgebras}
\label{ss:symplecticandorthogonal}

Fix $N > 0$ an integer and $\ve = \pm 1$ such that $\ve^N = 1$. We keep the partition $\lambda$ and shift matrix $\sigma$ from Section~\ref{ss:Dynkinandcentraliser}, and we now also assume that $\lambda \in \P_\ve(N)$ and that $\nu$ is $(-\ve(-1)^{\lambda_1}, \sigma)$-admissible.

Consider the matrix 
\begin{eqnarray}
\label{e:Jdefinition}
J:=  \sum_{a=1}^{m} \sum_{i=1}^{\nu_a} \sum_{j=1}^{\lambda(\nu)_a} (\beta)^i (-1)^{j+1}  e_{a,i,j;a,\nu_a + 1 - i, \lambda(\nu)_a + 1 - j}
\end{eqnarray}
This is a block diagonal matrix can be described as follows. For each index $a = 1,...,m$ there is a block of size $\nu_{a}\lambda(\nu)_a$, and each of these blocks is an antidiagonal matrix with entries $\pm 1$. Furthermore, when $\ve = 1$ the matrix $J$ is symmetric and when $\ve = -1$ it is anti-symmetric.

Now define an involution of $\g = \gl_N$ by the rule
\begin{eqnarray}
\label{e:taudefinition}
\tau : X \mapsto - J^{-1} X^\top J
\end{eqnarray}

We decompose $\g = \g_0 \oplus \g_1$ into eigenspaces for $\tau$ so that $\g_0$ is the space of $\tau$-invariants. Note that $J$ is symmetric  and $\g_0 \cong \so_N$ for $\ve =1$, whilst $J$ is antisymmetric and $\g = \sp_N$ for $\ve = -1$.

By Lemma~\ref{L:tauaction}(ii) we have $\{e,h\}\subseteq \g_0$, whilst $f\in \g_0$ thanks to \cite[Lemma~3.4.4]{CM93}. Write $G_0$ for the connected component of the subgroup of $\GL_N$ consisting of elements which preserve the bilinear form associated to $J$, so that $G_0$ is connected and $\Lie(G_0) = \g_0$.

\begin{Lemma}
\label{L:tauaction}
For all admissible indices we have: 
\begin{enumerate}
\setlength{\itemsep}{4pt}
\item[(i)] $\tau(e_{a,i,j;b,k,l}) = (\beta)^{i+k}(-1)^{j +l + 1}  e_{b,\nu_b + 1 - k, \lambda(\nu)_b + 1 - l; a,\nu_a + 1 - i, \lambda(\nu)_a + 1 - j}$;
\item[(ii)] $\tau(e) = e$ and $\tau(h)=h$;
\item[(iii)] $\tau(c_{a,i,b,k}^{(r)}) = (\beta)^{i+k}(-1)^{r-\frac{\lambda_k - \lambda_i}{2} - 1}c_{b,\nu_b + 1 - k,a, \nu_a + 1 - i}^{(r)}.$

\end{enumerate}
\end{Lemma}
\begin{proof}
Part (i) is proven using the fact that $J^{-1} = \epsilon J$, and multiplying matrices, bearing in mind the fact that the $\lambda_i$ all have the same parity. 
Now (ii) and (iii) follow by applying (i) to \eqref{e:edefinition}, \eqref{e:hdefinition} and \eqref{e:cijrdefinition}.
\end{proof}

\begin{Remark}
From henceforth the symbolic complexity will get somewhat more complicated, and to ease notation we introduce a certain involution on our index sets. Whenever we have a four-tuple of indices $(a,i,b,k)$ in the ranges $1 \le a,b \le m, \ 1 \le i \le \nu_a, \ 1 \le k \le \nu_b)$ we will write
\begin{eqnarray}
(a,i',b,k') := (a,\nu_a+1-i,b,\nu_b+1-k).
\end{eqnarray}
Similarly for a six-tuple $(a,i,j;b,k,l)$ with indexes in the range $1\le a,b \le m,\ 1\le i\le \nu_a, \ 1\le k \le \nu_b,\ 1\le j\le \lambda(\nu)_a, \ 1\le l \le \lambda(\nu)_b)$ so that, for example, Lemma \ref{L:tauaction} (i) could be read as $$\tau(e_{a,i,j;b,k,l}) = (\beta)^{i+k}(-1)^{j +l + 1}  e_{b, k', \lambda(\nu)_b + 1 - l; a,i', \lambda(\nu)_a + 1 - j}.$$
\end{Remark}

\subsection{Generators of the classical finite $W$-algebra}
\label{ss:generatorsWalgebra}

Here we reproduce formulas for the Poisson generators of $S(\gl_N, e)$ discovered by Brundan and Kleshchev in their seminal work \cite[\textsection 9]{BK06}, in which they gave presentations of the non-commutative finite $W$-algebras associated to $\gl_N$. The results of this section generalise \cite[\textsection 4.3]{To23} to the parabolic setting. Retain all of the choices and notation from Section~\ref{ss:symplecticandorthogonal}.

Now for $1 \le a,b \le m$, \ $1 \le i \le \nu_a$, \ $1 \le k \le \nu_b$, \ $0 \le x < n$ and $r > 0$, we let
\begin{equation}\label{e:invaraintsdef}
t_{a,i,b,k;x}^{(r)} := \sum_{s=1}^r (-1)^{r-s} \sum_{\substack{(a_u,i_u, j_u,b_u, k_u, l_u) \\ u =1,...,s}} (-1)^{\hash(k,x)}
e_{a_1,i_1, j_1; b_1,k_1, l_1} \cdots e_{a_s,i_s, j_s;b_s, k_s, l_s} \in S(\g(\geez))
\end{equation}
where we use the symbol
 \begin{eqnarray}
 \hash(k,x):={\hash \{q=1,\dots,s-1 \mid \bnu_{b_{q}-1}+ k_q \le x\}}.
 \end{eqnarray}
and the second summation is taken over all $6$-tuples of indices $(a_u,i_u, j_u,b_u, k_u, l_u)_{u=1,...,s}$ in the ranges
\begin{eqnarray}
\label{e:indexbounds}
\begin{array}{c}
1 \le a_u, b_u \le m, \ \ 1\le i_u \le \nu_{a_u}, \ \  1\le k_u \le \nu_{b_u} , \\  1\le j_u \le \lambda(\nu)_{a_u} \ \ \text{ and } \ \  1\le l_u \le \lambda(\nu)_{b_u}
\end{array}
\end{eqnarray}
satisfying the following six conditions:
\begin{itemize}
\setlength{\itemsep}{6pt}
\item[(a)] $\sum_{u=1}^s (2l_u - 2j_u) + \lambda(\nu)_{a_u} - \lambda(\nu)_{b_u} = 2(r - s)$;
\item[(b)] $2l_u- 2j_u + \lambda(\nu)_{a_u} - \lambda(\nu)_{b_u} \ge 0$ for each $u = 1,\dots,s$;
\item[(c)] if $\bnu_{b_u-1}+k_u > x$, then $l_u < j_{u+1}$ for each $u = 1,\dots,s-1$;
\item[(d)] if $\bnu_{{b_u}-1}+k_u \le x$ then $l_u \ge j_{u+1}$ for each $u = 1,\dots,s-1$; 
\item[(e)] $a_1 = a$, $i_1 = i$, $b_s = b$ and $k_s = k$.
\item[(f)] $b_u = a_{u+1}$ and $k_u = i_{u+1}$ for all $u=1,...,s-1$.
\end{itemize}

Fixing $r > 0$ and choosing $s \in \{1,...,r\}$ we make the notation $\X_{a,i,b,k; x}^{(r, s)}$ for the set of ordered sets $(a_u,i_u, j_u,b_u, k_u, l_u)_{u = 1}^s$ in the range \eqref{e:indexbounds} satisfying the conditions (a)--(f) and consider the map 
\begin{eqnarray*}
\begin{array}{l}
 
 \upsilon(a_u,i_u,j_u,b_u,k_u,l_u)_{u=1,...,s} :=\vspace{6pt}
 \\ \hspace{-20pt} (b_{s+1-u}k'_{s+1-u}, \lambda(\nu)_{b_{s+1-u}} + 1 - l_{s+1-u}, a_{s+1-u}, i'_{s+1-u}, \lambda(\nu)_{a_{s+1-u}} + 1 - j_{s+1-u})
 \end{array}
\end{eqnarray*}
It is clear that this is an involution on collections of ordered sets in the range \eqref{e:indexbounds}, and we now prove that it respects the conditions (a)-(f).
\begin{Lemma}
\label{L:upsilonbijects}
For all $t=1,...,s$ the map $\upsilon$ defines a bijection 
\begin{eqnarray}
\label{e:Xbijection}
\X_{a,i,b,k; \bnu_t}^{(r,s)} \overset{1\text{-}1}{\longrightarrow} \X_{b, k', a, i'; \bnu_t}^{(r,s)}.
\end{eqnarray}
\end{Lemma}
\begin{proof}
The index bounds \eqref{e:indexbounds} ensure that $\X_{a,i,b,k, \bnu_t}^{(r,s)}$ is finite. Since $\upsilon^2$ is the identity map it will suffice to show that $\upsilon$ maps the left hand side of \eqref{e:Xbijection} to the right. We fix $(a_u,i_u, j_u,b_u k_u, l_u)_{u = 1,...,s}\in \X_{a,i,b,k; \bnu_t}^{(r,s)}$ and check that conditions (a)--(f) hold for $\upsilon(a_u,i_u, j_u, b_u, k_u, l_u)_{u=1,...,s}$. For example, 
\begin{eqnarray*}
\sum_{u=1}^s \big(2(\lambda(\nu)_{a_{s+1-u}} + 1 - j_{s+1-u}) - 2(\lambda(\nu)_{b_{s+1-u}} + 1 - l_{s+1-u}) + (\lambda(\nu)_{b_{s+1-u}} - \lambda(\nu)_{a_{s+1-m}})\big)\ \ \ \ \ \ \ \ \\
\ \ \ \ \ \ \ \ \ \ = \sum_{u=1}^s (2l_u - 2j_u + \lambda(\nu)_{a_u} - \lambda(\nu)_{b_u})   = 2(r - s)
\end{eqnarray*}
This shows that when (a) is satisfied for $(a_u,i_u, j_u, b_u, k_u, l_u)_{u=1,...,s}$, it follows that the formulation of (a) which defines $\X_{b, k' , a, i'; \bnu_t}^{(r,s)}$  holds for $\upsilon(a_u,i_u, j_u, b_u, k_u, l_u)_{u=1,...,s}$. Conditions (b)--(f) can be checked similarly.
\end{proof}

Now we define a map $\theta : \g^e \to S(\g(\geez))$ by
\begin{eqnarray}
\label{e:intermediatemap}
\begin{array}{rcl}
c_{a,i,a,i}^{(r-1)}  \longmapsto (-1)^{r-1}t_{a,i,a,j; \bnu_{a-1}}^{(r)} & & \text{ for } a=1,...,m \\
c_{a,i,b,k}^{(r-1)} \longmapsto (-1)^{r-1}t_{a,i,b,k; \bnu_{b-1}}^{(r)} & & \text{ for } 1\le a < b \le m \\
c_{b,k,a,i}^{(r-1)} \longmapsto (-1)^{r-1}t_{b,k,a,i; \bnu_{b-1}}^{(r)} & & \text{ for } 1\le a < b \le m 
\end{array}
\end{eqnarray}
for all admissible $i,j,r$.

By \eqref{e:gooddegrees} and Lemma~\ref{L:tauaction}(i) we see that $\tau$ induces an involution on $\g^e$ and $S(\g(\geez))$, and the latter preserves $S(\g,e)$. 
\begin{Proposition}
\label{P:tauequivariance}
$\theta : \g^e \to S(\g,e)$ is $\tau$-equivariant. 
\end{Proposition}
\begin{proof}
The assertion that the image of $\theta$ lies in $S(\g,e)$ is an immediate consequence of \cite[Theorem~10.1]{BK06}, upon taking the top graded term with respect to the Kazhdan filtration of the finite $W$-algebra.

Thanks to Lemma~\ref{L:tauaction}(iii) the $\tau$-equivariance will follow from \eqref{e:tauclaim3}.  Fix indices $1\le a, b \le m, 1 \le i \le \lambda(\nu)_a,  1 \le k \le \lambda(\nu)_b,   \ 0\le x < n$ and $0 < r$. Choose $s\in \{1,...,r\}$. Let $(a_m,i_m, j_m,b_m, k_m,l_m)_{m=1}^s \in \X_{a,i,b,k; \bnu_t}^{(r,s)}$. Thanks to Lemma~\ref{L:tauaction}(i) we have
\begin{eqnarray}
\label{e:tauclaim3}
\tau(\prod_{m=1}^s e_{a_m,i_m, j_m;b_m k_m, l_m}) &=& (\beta)^{\sum_{m=1}^s k_m+i_m}(-1)^{\sum_{m=1}^s (j_m - l_m + 1)} \prod_{t=1}^s e_{\upsilon(a_m,i_m, j_m;b_m k_m, s_m)}
\end{eqnarray}
Conditions (a), (e) (f) imply that
$\sum_m (j_m - l_m + 1) = \frac{\lambda_{i_1} - \lambda_{k_s}}{2} - r.$ Conditions (e), (f) imply $(-1)^{\sum_{m=1}^s k_m+i_m}=(-1)^{i+k}$.
Hence the sign on the right hand side of \eqref{e:tauclaim3} is equal to $(\beta)^{i+k}(-1)^{r-\frac{\lambda_k - \lambda_i}{2}}$. Furthermore condition (f) together with our condition on $x$ ensures that
$\hash \{ m =1,...,s-1 \mid \nu_1+\dots+ \nu_{b_m-1}+k_m \le \bnu_t\} = \hash \{m =2,...,s \mid \nu_1+\dots+ \nu_{a_m-1}+\lambda(\nu)_m + 1 - i_m \le \bnu_t\}.$
Combining these observations together with Lemma~\ref{L:upsilonbijects} we have
\begin{eqnarray*}
\tau(t_{a,i,b,k;\bnu_t}^{(r)}) &=& \sum_{s=1}^r (-1)^{r-s} \sum_{(a_m,i_m, j_m, b_m, k_m, l_m)_m \in \X_{a,i,b,k; \bnu_t}^{(r,s)}} (-1)^{\hash_{k,\bnu_t}}\tau (\prod_m e_{a_m,i_m, j_m, b_m, k_m, l_m}) \\
& = & (\beta)^{i+k}(-1)^{r-\frac{\lambda_k - \lambda_i}{2}} \sum_{s=1}^r (-1)^{r-s} \sum_{(a_m,i_m, j_m,b_m, k_m, l_m)_m \in \X_{b,k',a,i';\bnu_t}^{(r,s)}} (-1)^{\hash_{k,x}}\prod_m e_{a_m,i_m, j_m, b_m, k_m, l_m}\\
& = & (\beta)^{i+k}(-1)^{r-\frac{\lambda_k - \lambda_i}{2}} t_{b, k',a,i';\bnu_t}^{(r)}
\end{eqnarray*}
This completes the proof.
\end{proof}
\begin{Lemma}
For $ s_{a,b} < r \le s_{a,b} + \min(\lambda(\nu)_a, \lambda(\nu)_b)$ the linear parts of
\begin{eqnarray}
\begin{array}{rcl}
 (-1)^{r-1}t_{a,i,a,j; \bnu_{a-1}}^{(r)} & & \text{ for } a=1,...,m \\
  (-1)^{r-1}t_{a,i,b,k; \bnu_{b-1}}^{(r)} & & \text{ for } 1\le a < b \le m \\
  (-1)^{r-1}t_{b,k,a,i; \bnu_{b-1}}^{(r)} & & \text{ for } 1\le a < b \le m 
\end{array}
\end{eqnarray}
 respectively are 
 \begin{eqnarray}
\begin{array}{rcl}
c_{a,i,a,i}^{(r-1)}   & & \text{ for } a=1,...,m \\
c_{a,i,b,k}^{(r-1)}   & & \text{ for } 1\le a < b \le m \\
c_{b,k,a,i}^{(r-1)}   & & \text{ for } 1\le a < b \le m 
\end{array}
\end{eqnarray}
In particular $\theta$ fulfills the properties of \cite[Theorem 2.4]{To23}.
\end{Lemma}
\begin{proof}
The linear part of \ref{e:invaraintsdef} corresponds to $s=1$. Conditions (b)-(d) as well as condition (f) on the indices in  \ref{e:invaraintsdef} become redundant when $s=1$ whereas combining conditions (a) and (e) gives the sum
\begin{eqnarray}
(-1)^{r-1}\sum_{ 2r-1 = 2(l-j) + \lambda(\nu)_{a}- \lambda(\nu)_{b}} e_{a,i,j;b,k,l}
\end{eqnarray} which is equal to $(-1)^{r-1}c_{a,i,b,k}^{(r-1)}$ by definition \ref{e:cijrdefinition}.
\end{proof}

\section{Dirac reduction for shifted Yangians}
\label{s:DracredcutionforshiftedYangians}

\subsection{Parabolic presentations for shifted current Lie--Poisson algebras} 
\label{ss:CSpresentationsforshifts}
In this section we describe the parabolic Poisson presentations of shifted current Lie algebras. This is simultaneously a generalisation of \cite[\textsection 3.2]{To23}, which implicitly treated the shape $(1^n)$, and is a semiclassical analogue of the results of \cite{BK05}.

We will repeatedly use the notion of a presentation of a Poisson algebra which was briefly explained before the statement of Proposition~\ref{P:someprop}.

Throughout this section we fix $n > 0$ and a symmetric $n\times n$ shift matrix $\sigma$. Pick an admissible shape $\nu = (\nu_1,...,\nu_m)$ for $\sigma$, and recall the relative shift matrix $\sigma(\nu)$, as well as the notation $\bnu_a = \sum_{i=1}^a \nu_i$, introduced in Section~\ref{ss:admissibleshapes}.


The {\it current algebra} is the Lie algebra $\cc_n := \gl_n \otimes \C[t]$. We write $e_{i,j}$ for standard matrix units and $e_{i,j}t^r := e_{i,j} \otimes t^r$. A basis is given by $\{e_{i,j} t^r \mid 1\leq i , j \leq n, \ r \ge 0\}$.

We will need a parabolic labelling of the basis of $\cc_n(\sigma)$. Define elements
\begin{eqnarray}
e_{a,i,b,j}t^r := e_{\bnu_{a-1} + i, \bnu_{b-1} + j} t^r.
\end{eqnarray}
We warn the reader that the notation used here for matrix units is different from that appearing in Section~\ref{ss:Dynkinandcentraliser}.

We define the {\it shifted current algebra} $\cc_n(\sigma)$ to be the subalgebra spanned by 
\begin{eqnarray}
\label{e:shiftedcurrentgens}
\{e_{a,i,b,j}t^r \mid 1\le a,b\le m, \ 1 \le i \le \nu_a, \ 1 \le j \le \nu_b, \ s_{a,b} \le r \}.
\end{eqnarray}
\begin{Lemma}
\label{L:utgenslemma}
The Lie subalgebra $^{\nu}\u_n(\sigma) \subseteq \cc_n(\sigma)$ which is spanned by elements \eqref{e:shiftedcurrentgens} with $a< b$, has generators
\begin{eqnarray}
\label{e:utcurrentsgens}
\{e_{a,i,j; r}  \mid 1\le i < m, 1 \le i \le \nu_{a}, \ 1\le j \le \nu_{a+1}, \ s(\nu)_{a,a+1} \le r \}
\end{eqnarray}
and relations
\begin{eqnarray}
\label{e:utcurrentrels1}
\big[e_{a,i,j; r}, e_{b,h,k; s}\big] &=& 0 {\text{ \ \ \ for \ \ \ } |i-j| \ne 1\text{ or if } b=a+1 \text{ and } h \neq j,}
\end{eqnarray}
\begin{eqnarray}
\label{e:utcurrentrels2}
{}\ \ \big[e_{a,i,j; r}, e_{a+1,h,k; s+1}\big] - \big[e_{a,i,j; r+1}, e_{a+1,h,k; s}\big] &=&0,\vspace{8pt}\\
\label{e:utcurrentrels3}
{}\Big[e_{i; r_1} \big[e_{i; r_2}, e_{j;r_3} \big]\Big] + \Big[e_{i; r_2} \big[e_{i; r_1}, e_{j;r_3} \big]\Big] &=& 0 \text{ \ \ \ for all } |i-j| = 1.
\end{eqnarray}
\end{Lemma}
\begin{proof}
Write $\underline{0}$ for the $n\times n$ zero matrix. It follows from \eqref{e:shiftmatrixdefn} that the linear map $^{\nu}\u_n(\sigma) \to ^{\nu}\u_n(\underline{0})$ defined by $e_{a,i,b,j}t^r \mapsto e_{a,i,b,j}t^{r-s_{a,b}}$ is a Lie algebra isomorphism and so it suffices to prove the current lemma when $\sigma = \underline{0}$.

Let $^{\nu}\h\u_n$ be the Lie algebra with generators \eqref{e:utcurrentsgens} and relations \eqref{e:utcurrentrels1}--\eqref{e:utcurrentrels3}, with $\sigma = \underline{0}$. We inductively define elements $e_{a,i,b,j;r} \in {}^{\nu}\h\u_n$ by setting $e_{a,i,a+1,j;r} := e_{a,i,j;r}$ and $e_{a,i,b,j;r} := [e_{a,i,k,b-1; r}, e_{b-1,k,j ; 0}]$ for $1\le a < b \le m$ and all admissible $i,j$. There is a homomorphism $^{\nu}\h\u_n \onto {}^{\nu}\u_n(\underline 0)$ given by $e_{a,i,j;r} \mapsto e_{a,i,a+1,j}t^r$ and, in order to show that it is an isomorphism, we show that $^{\nu}\h\u_n$ is spanned by the elements $\{e_{a,i,b,j;r} \mid 1\le a<b\le m, \ 1 \le i \le \nu_a, \ 1\le j \le \nu_b \ 0 \le r\}$.

Following (1)--(7) in the proof of \cite[Lemma~6.7]{BK05} verbatim 
we have for all $a,b,c,d,j,k,l,r,s$
\begin{eqnarray}
\label{e:haturelations}
[e_{a,i,b,j;r} , e_{c,k,d,l;s}] = \delta_{b,c}\delta_{j,k} e_{i,l;r+s} - \delta_{a,d}\delta_{i,l} e_{k,j;r+s}.
\end{eqnarray}

Define an ascending filtration on $\h\u_n = \bigcup_{d > 0} \F_d \h\u_n$ satisfying $\F_{d} \h\u_n = \sum_{\substack{d_1 + d_2 = d}} [\F_{d_1} \h\u_n, \F_{d_2} \h\u_n]$  by placing $e_{a,i,j;r}$ in degree 1. We prove by induction that $\F_d\h\u_n$ is spanned by elements $e_{a,i,b,j;r}$ with $j-i \le d$. The base case $d = 1$ holds by definition. For $d_1 + d_2 = d > 1$ we know by the inductive hypothesis that $\F_{d_1} \h\u_n$ and $\F_{d_2} \h\u_n$ are spanned by elements $e_{a,i,b,j;r}$. Using \eqref{e:haturelations} we complete the induction, which finishes the proof. 
\end{proof}

Since $\cc_n(\sigma)$ is a Lie algebra the symmetric algebra $S(\cc_n(\sigma))$ carries a natural Poisson structure.

\begin{Theorem}
\label{T:CSshiftedcurrents1}
Let $\sigma$ be a symmetric shift matrix $\nu=(\nu_1,\dots,\nu_m)$ an admissible shape and let $\cc_n(\sigma)$ denote the shifted current algebra. Then $S(\cc_n(\sigma))$ is Poisson generated by
\begin{eqnarray}
\label{e:shiftedcurrentfreegens}
\begin{array}{l}
\{d_{a,i,j;r} \mid 1\le a \le m, 1\le i,j \le \nu_a, \ 0\le r\}\\ \hspace{60pt}  \cup \{e_{a,i,j;r} \mid 1\le a < m, 1\le i < \nu_a,1\le j < \nu_{a+1}, \ s_{i,i+1} \le r\} \\ \hspace{120pt} \cup \{f_{a,i,j;r} \mid 1\le a < m, 1\le i < \nu_{a+1},1\le j < \nu_{a}, \ s_{i+1,i} \le r\}  \end{array}
\end{eqnarray}
subject to relations
\begin{eqnarray}
\setlength{\itemsep}{4pt}
\label{e:shiftedcurrentfreerels1}
& & \big\{d_{a,i,j;r}, d_{b,h,k;s}\big\} = \delta_{a,b} (\delta_{h,j}d_{a,i,k;r+s}-\delta_{i,k}d_{a,h,j;r+s})\\
\label{e:shiftedcurrentfreerels3.5}
& & \big\{e_{a,i,j; r}, f_{b,h,k;s}\big\} = \delta_{a,b}( \delta_{h,j}d_{a,i,k; r+s} - \delta_{i,k}d_{a+1,h,j; r+s})\\
\label{e:shiftedcurrentfreerels2}
& & \big\{d_{a,i,j;r}, e_{b,h,k;s}\big\} = \delta_{a,b}\delta_{h,j}e_{a,i,k;r+s} - \delta_{a,b+1}\delta_{i,k} e_{b,h,j; r+s},\\
\label{e:shiftedcurrentfreerels3}
& & \big\{d_{a,i,j;r}, f_{b,h,k;s}\big\} = \delta_{a,b+1}\delta_{h,j}f_{b,i,k;r+s} - \delta_{a,b}\delta_{i,k} f_{a,h,j; r+s},\\
\label{e:shiftedcurrentfreerels4}
&  &\big\{e_{a,i,j; r}, e_{a+1,h,k; s + 1}\big\} - \big\{e_{a,i,j; r+1}, e_{a+1,h,k; s}\big\} = 0,\\
\label{e:shiftedcurrentfreerels5}
& & \big\{f_{a,i,j; r}, f_{a+1,h,k; s + 1}\big\} - \big\{f_{a,i,j; r+1}, f_{a+1,h,k; s}\big\} = 0,\\
\label{e:shiftedcurrentfreerels6}
& & \big\{e_{a,i,j;r}, e_{b,h,k;s}\big\} = 0 \text{ for } |i-j| \ne 1 \text{ or if } b=a+1 \text{ and } h \neq j,\\
\label{e:shiftedcurrentfreerels7}
& & \big\{f_{a,i,j;r}, f_{b,h,k;s}\big\} = 0 \text{ for } |i-j| \ne 1 \text{ or if } b=a+1 \text{ and } i \neq k,\\
\label{e:shiftedcurrentfreerels8}
& & \Big\{e_{a,i,j; r_1}, \big\{e_{a,h,k; r_2}, e_{b,f,g;r_3}\big\}\Big\} + \Big\{e_{a,i,j; r_2}, \big\{e_{a,h,k; r_1}, e_{b,g,f;r_3}\big\}\Big\} = 0 \text{ for } |a-b| = 1,\\
\label{e:shiftedcurrentfreerels9}
& & \Big\{f_{a,i,j; r_1}, \big\{f_{a,h,k; r_2}, f_{b,g,f; r_3}\big\}\Big\} + \Big\{f_{a,i,j; r_2}, \big\{f_{a,h,k; r_1}, f_{b,g,f; r_3}\big\}\Big\} = 0 \text{ for } |a-b| = 1.
\end{eqnarray}
\end{Theorem}

\begin{proof}
By \cite[(3.1)]{To23} it suffices to show that $\cc_n(\sigma)$ is generated as a Lie algebra by \eqref{e:shiftedcurrentfreegens} subject to relations \eqref{e:shiftedcurrentfreerels1}--\eqref{e:shiftedcurrentfreerels9}. Let $(\hcc_n(\sigma), \{\cdot, \cdot\})$ denote the Lie algebra with these generators and relations. Define a map from the set \eqref{e:shiftedcurrentfreegens} to $\cc_n(\sigma)$ by $d_{a,i,j;r} \mapsto e_{a,i,a,j}t^r, \ e_{a,i,j;r} \mapsto e_{a,i,a+1,j}t^r, \ f_{a,i,j;r} \mapsto e_{a+1,i, a,j} t^r$.
One can easily verify using \eqref{e:haturelations} that this extends to a surjective Lie algebra homomorphism $\h\cc_n(\sigma) \onto \cc_n(\sigma)$. To show that this is an isomorphism it suffices to show that the elements
\begin{eqnarray}
\label{e:hSgens}
\{e_{a,i,b,j;r} \mid 1 \le a,b\le m, \ 1 \le i \le \nu_a, \ 1 \le b \le \nu_{b} \ s_{a,b}\le r\} \subseteq \hcc_n(\sigma)
\end{eqnarray}
defined inductively by setting $e_{a,i,a+1,j; r} := e_{i,r}$, $e_{a+1, i,a,j; r} := f_{a,i,j;r}$ and
\begin{eqnarray}
\label{e:shiftedcurrenthighereij}
& e_{a,i,b,j;r} := \{e_{a,i,b-1,k;r-s_{b-1,b}}, e_{b-1,k,j; s_{b,b-1}}\}  & \text{ for } a < b;\\
\label{e:shiftedcurrenthigherfij}
& e_{a,i,b,j;r} := \{f_{a-1,i,k; s_{a,a-1}}, e_{a-1,k,b, j; r-s_{a,a-1}}\}  & \text{ for } a > b.
\end{eqnarray}
form a spanning set. Using \eqref{e:shiftedcurrentfreerels1},\eqref{e:shiftedcurrentfreerels2}, \eqref{e:shiftedcurrentfreerels3} and a simple inductive argument one can see that $\h\cc_n(\sigma)$ is a direct sum of three subalgebras: the diagonal subalgebra, spanned by the elements $d_{a,i,j;r}$, and the upper and lower triangular subalgebras $^{\nu}\u_n^+(\sigma)$ and $^{\nu}\u_n^-(\sigma)$ generated by the elements $e_{a,i,j; r}$, respectively by the elements $f_{a,i,j; r}$.

In order to complete the proof it suffices to show that $\u_n^+(\sigma)$ and $\u_n^-(\sigma)$ are spanned by the elements defined in \eqref{e:shiftedcurrenthighereij} and \eqref{e:shiftedcurrenthigherfij} respectively. 
Since the argument is identical for $\u_n^+(\sigma)$ and $\u_n^-(\sigma)$ we only need to consider the former, where the claim follows from Lemma~\ref{L:utgenslemma} 
\end{proof}

The Dirac reduced algebra in the following theorem should be viewed as a {\it twisted shifted current Lie--Poisson algebra} (see \cite[Remark~2.3]{To23}).
Recall that for a shape $\nu=(\nu_1, \dots, \nu_m)$ and a four-tuple $(a,i,b,j)$ $1 \le,a,b \le m, \ 1 \le i \le \nu_a,\ 1 \le j \le \nu_b$ we write $(a,i',b,j'):=(a,\nu_a+1-i,b,\nu_b+1-j)$.
\begin{Theorem}\label{T:twistedcurrentalgpresenrtation}
\label{dirredcurrent}
If $\beta = \pm 1$ and $\nu$ is a $(\beta, \sigma)$-admissible shape then $S(\cc_n(\sigma))$ admits an involutive Poisson automorphism
\begin{eqnarray}
\label{e:tauonshiftedcurrents}
\tau : e_{a,i,b,j}t^r \mapsto (\beta)^{i+j}(-1)^{r - 1 + s(\nu)_{a,b}} e_{b, j' ,a,i'}t^r.
\end{eqnarray}
The Dirac reduction $R(S(\cc_n(\sigma), \tau) = S(\cc_n(\sigma)^\tau)$ is Poisson generated by
\begin{align}
\label{e:shiftedcurrentinvariantsfreegen}
\begin{array}{l}
\{\eta_{a,i,j;r} \mid 1 \le a \le m, \ 1 \le i,j \le \nu_{a}, \ r \geq 0 \} \\
 \hspace{40pt} \cup \{\theta_{a,i,j;r} \mid 1 \le a \le m-1, \ 1 \le i \le \nu_{a}, \ 1 \le j \le \nu_{a+1}, \ r \geq s_{a,a+1} \}
 \end{array}
\end{align}
subject to the relations
\begin{align}\label{currel0}
\eta_{a,i,j,r}=(\beta)^{i+j}(-1)^{r-1}\eta_{a,j',i',r}
\end{align}
\begin{align} \label{currel1}
\left\{\eta_{a,i,j;r},\eta_{b,h,k;s} \right\}=\delta_{a,b} \frac{1}{2}(\delta_{h,j}\eta_{a,i,k;r+s}-\delta_{i,k}\eta_{a,h,j;r+s}+(\beta)^{h+k}(-1)^{s-1}(\eta_{a,i,h';r+s}\delta_{k',j}-\delta_{i,h'}\eta_{a,k',j;r+s}))
\end{align}
\begin{align}\label{currel2}
\left\{\eta_{a,i,j;r},\theta_{b,h,k;s} \right\}&=\frac{1}{2}( \delta_{a,b}\delta_{h,j} \theta_{a,i,k;r+s}-\delta_{a,b+1}\delta_{i,k}\theta_{b,h,j;r+s})\\
&+(\beta)^{i+j}\frac{1}{2}(-1)^{r-1}(\delta_{a,b}\delta_{h,i'}\theta_{a,j',k;r+s}-\delta_{a,b+1}\delta_{k',j}\theta_{b,h,i';r+s}) 
\end{align}
\begin{align}\label{currel3}
    \left\{\theta_{a,i,j;r},\theta_{a,h,k;s} \right\}=\frac{1}{2}(\beta)^{h+k}(-1)^{s_{a,a+1}+s-1}(\delta_{i,h'}\eta_{a+1,k',j;r+s}-\delta_{k,j'}\eta_{a,i,h',r+s}) \textnormal{ for} \ r < s
\end{align}
\begin{align}\label{currel4}
    \{\theta_{a,i,j;r}, \theta_{a+1,h,k;s}\}-\{\theta_{a,i,j;r+1}, \theta_{a+1,h,k;s}\}=0
\end{align}
\begin{align}\label{currel5}
  \{\theta_{a,i,j;r}, \theta_{b,h,k;s}\}=0  \textnormal{ for }  b>a+1  \textnormal{ or if } b=a+1 \textnormal{ and } h \neq j 
\end{align}
\begin{align}\label{currel6}
        \{\theta_{a,i,j;r}, \{\theta_{a,h,k;s}, \theta_{b,f,g;t}\}\} +\{\theta_{a,i,j;s}, \{\theta_{a,h,k;r}, \theta_{b,f,g;t}\}\}=0 \textnormal{ if } \mid a-b \mid =1 \textnormal{ and } r+s \textnormal{ odd}
\end{align}
\begin{align}\label{currel7}
\begin{array}{l}
    \{\theta_{a,i,j;r}, \{\theta_{a,h,k;s}, \theta_{b,f,g;t}\}\} +\{\theta_{a,i,j;s}, \{\theta_{a,h,k;r}, \theta_{b,f,g;t}\}\}\vspace{6pt}\\
    \hspace{80pt} =\frac{1}{2}(\beta)^{i+j}(-1)^{r+s_{a,a+1}}\left( \delta_{a+1,b}\delta_{f,k}\delta_{h,i'}\theta_{a+1,j',g;r+s+t}+\delta_{a,b+1}\delta_{h,g}\delta_{j',k}\theta_{b,f,i';r+s+t} \right) \vspace{10pt} \\
    \hspace{120pt}\textnormal{ if } \mid a-b \mid =1 \textnormal{ and } r+s \textnormal{ even}.
    \end{array}
\end{align}
\end{Theorem}
\begin{proof}
Relation \eqref{e:haturelations} implies that $\tau$ gives a Poisson automorphism. By \cite[Remark~2.3]{To23} we can identify $R(S(\cc_n(\sigma), \tau)$ with $S(\cc_n(\sigma)^\tau)$, so it suffices to check that $S(\cc_n(\sigma)^\tau)$ has the stated Poisson presentation. Use \cite[(3.1)]{To23} to reduce the claim to a statement about the presentation of $\cc_n(\sigma)^\tau$.

Consider the Lie algebra $\hcc_n(\sigma)^\tau$ which is generated by the set \eqref{e:shiftedcurrentinvariantsfreegen} subject to relations \eqref{currel1}--\eqref{currel7}. Make the notation $\theta_{a,i,a+1,j}t^r := e_{a,i,a+1,j}t^r + \tau(e_{a,i,a+1,j}t^r), \ \eta_{a,i,a,j}t^r := e_{a,i,a,j}t^r + \tau(e_{a,i,a,j}t^r)  \in \cc_n(\sigma)^\tau$. We define a map from the set \eqref{e:shiftedcurrentinvariantsfreegensproof} to $\cc_n(\sigma)^\tau$ by sending $\eta_{a,i,j;r} \mapsto \eta_{a,i,a,j}t^{r} \in\cc_n(\sigma)^\tau$, and sending $\theta_{a,i,j;r} \mapsto \theta_{a,i,a+1,j}t^r \in\cc_n(\sigma)^\tau$. One can check that this determines a surjective Lie algebra homomorphism
$\hcc_n(\sigma)^\tau \onto \cc_n(\sigma)^\tau$, indeed, checking that relations \eqref{currel1}--\eqref{currel7} hold amongst the corresponding elements of $\cc_n(\sigma)^\tau$ is a routine calculation using \eqref{e:haturelations} and relation \ref{currel0} follows immediately by definition of $\tau$.

In order to complete the proof of (2) it is sufficient to show that this map is an isomorphism.
This proof is similar to \cite[Theorem 3.3]{To23} but in our case there is a linear relation between the $\eta_{a,i,j;r}$.
 For all admissible $a,b,i,j,$ and $r \geq s_{a,b}$ we inductively define elements
\begin{eqnarray}
\label{e:someotherlabel}
\theta_{a,b,i,j;r} := \{\theta_{a,i,a+1,k;s_{i,i+1}}, \theta_{a+1,k,b, j; r-s_{i,i+1}}\} \in \hcc_n(\sigma)^\tau. 
\end{eqnarray}
where $\theta_{a,i,a+1,j;r} := \theta_{a,i,j;r}$. Let $\{v_1, v_2,\dots \}$ be a subset of  $\{\eta_{a,i,j;r} \mid 1 \le a \le m, \ 1 \le i,j \le \nu_{a}, \ r \geq 0 \}$ such that it is a basis  for the vector space spanned by $\{\eta_{a,i,j;r} \mid 1 \le a \le m, \ 1 \le i,j \le \nu_{a}, \ r \geq 0 \}$ together with the linear relation \ref{currel0}.
It remains to check that $\hcc_n(\sigma)^\tau$ is spanned by 

\label{e:shiftedcurrentinvariantsfreegensproof}
\begin{align}\label{e:shiftedcurrentinvariantsfreegensproof}
\begin{array}{ll}
&\{v_1,v_2, \dots \} \cup \{\theta_{a,i,b,j;r} \mid 1 \le a,b \le m, \ 1 \le i \le \nu_{a}, \ 1 \le j \le \nu_{b}, \ r \geq s_{a,b} \}.
\end{array}
\end{align}

We define a filtration $\hcc_n(\sigma)^\tau = \bigcup_{i > 0} \F_i \hcc_n(\sigma)^\tau$ by placing the generators \eqref{e:shiftedcurrentinvariantsfreegen} in degree 1 and satisfying $\F_{d} \h\cc_n(\sigma)^\tau = \sum_{d_1 + d_2 = d} \{\F_{d_1} \hcc_n(\sigma)^\tau, \F_{d_2} \hcc_n(\sigma)^\tau \}$. By convention $\F_{0} \hcc_n(\sigma)^\tau = 0$. The associated graded Lie algebra $\gr \hcc_n(\sigma)^\tau = \bigoplus_{i > 0} \F_i \hcc_n(\sigma)^\tau / \F_{i-1} \hcc_n(\sigma)^\tau$ is generated by elements 
\begin{eqnarray}
& & \overline{\eta}_{a,i,j;r} := \eta_{a,i,j;r} + \F_{0} \hcc_n(\sigma)^\tau \text{ for } 1\le a \le m, \ 1 \le i,j \le \nu_a \ 0 \le r \\
\label{e:somelabel}
& & \overline{\theta}_{a,i,j;r} := \theta_{a,i,j;r} + \F_0 \hcc_n(\sigma)^\tau \text{ for } 1\le a < m-1, \ 1 \le i \le \nu_a, \ 1 \le j \le \nu_{a+1}, \  \ s_{a,a+1} \le r.
\end{eqnarray}
These generators of $\gr \hcc_n(\sigma)^\tau$ satisfy the top graded components of the relations \eqref{currel0}--\eqref{currel7}. 

Let $\af$ be an abelian Lie algebra with spanning set $\{d_{a,i,j} t^r \mid 1\le a \le m, \ 1 \le i,j \le \nu_{a}, \ 0 \leq r\}$ and relations $d_{a,i,j} t^r=(\beta)^{i+j}(-1)^{r-1}d_{a,\nu_a+1-j,\nu_a+1-i}$. Let $\u_n(\sigma)$ be the Lie subalgebra of $\cc_n(\sigma)$ described in Lemma~\ref{L:utgenslemma}. Then $\af \oplus \u_n(\sigma)$ is a Lie algebra with $\af$ an abelian ideal. Comparing the top components \eqref{currel0}--\eqref{currel7} to \eqref{e:utcurrentrels1}--\eqref{e:utcurrentrels3} we see that there is surjective Lie algebra homomorphism $\af \oplus \u_n(\sigma) \onto \gr \hcc_n(\sigma)^\tau$ defined by $d_{a,i,j} t^{r} \mapsto \overline{\eta}_{a,i,j; r}$ and $e_{a,i,a+1,j}t^r \mapsto \overline{\theta}_{a,i,j;r}$. The algebra $\af \oplus \u_n(\sigma)$ has basis consisting of elements $e_{a,h,b,k}t^s$ where $ 1\le a < b \le m, 1 \le h \le \nu_{a}, \ 1 \le k \le \nu_{b},  \ s \ge s_{a,b}$ together with a basis $\{w_1,w_2,\dots\}$ for the vector space $\mathfrak{a}$ which can be picked such that $w_1 \mapsto v_1$. It follows that $\gr \hcc_n(\sigma)^\tau$ is spanned by elements $\overline{v_1},\overline{v_2},\dots, \overline{\theta}_{a,i,b,j;s}$ where the indices vary in the ranges specified in \eqref{e:shiftedcurrentinvariantsfreegensproof}, and $\overline{\theta}_{a,i,b,j;s}$ defined inductively from \eqref{e:somelabel}, analogously to \eqref{e:someotherlabel}. We deduce that $\hcc_n(\sigma)^\tau$ is spanned by the required elements, which completes the proof.
\end{proof}
\begin{Remark}
It is crucial for this theorem that our shape $\nu$ is $(\beta,\sigma)$-admissible, otherwise $\tau$ is not an involution.
\end{Remark}

Our next result, which we record here for later use, explains the relationship between the shifted current algebras and the centralisers described above. 
\begin{Lemma}
Let $\ve = \pm 1$ and $N > 0$ such that $\ve^N = 1$. Let $\lambda \in \P_\ve(N)$ be an even partition, and let $\sigma$ be the shift matrix determined by $\lambda$, as in Lemma~\ref{L:lambdabij}. Pick a $(-\ve(-1)^{\lambda_1}, \sigma)$-admissible shape and let $\tau : \cc_n(\sigma) \to \cc_n(\sigma)$ be the the involution  \eqref{e:tauonshiftedcurrents}. Abusing notation we also write $\tau : \gl_N \to \gl_N$ for the automorphism described in \eqref{e:taudefinition}.
\label{L:shiftedcurrentandcentraliser}
\begin{enumerate}
\setlength{\itemsep}{4pt}
\item There is a surjective Lie algebra homomorphism
\begin{eqnarray*}
\cc_n(\sigma) & \longtwoheadrightarrow & \gl_N^e  \\
e_{a,i,b,j}t^r & \longmapsto & c_{a,i,b,j}^{(r)}
\end{eqnarray*}
 where $c_{a,i,b,j}^{(r)} := 0$ for $r \ge s_{a,b} + \min(\lambda_a, \lambda_b)$.
The kernel is Poisson generated by $e_{1,i,1,j}t^r$ with $r\ge \lambda_1$, $1 \leq i,j \leq \nu_{1}$ 
\item The homomorphism from (1) restricts to $\cc_n(\sigma)^\tau \onto (\g^e)^\tau$.
\begin{itemize}
\item[(i)] For $\ve=1$ the kernel is generated as a Lie ideal by 
\begin{eqnarray}
\label{e:LAidealgens}
\{\eta_{1,i,j;r} \mid r \ge \lambda_1 \ , 1 \le i,j \le \nu_1\}
\end{eqnarray}
\item[(ii)]For $\ve= -1$ and $\nu_1\ne 1$  the kernel is generated by \eqref{e:LAidealgens}
\item[(iii)] For $\ve =-1$ and $\nu_1 = 1$ the kernel is generated by those elements,  along with
\begin{eqnarray}
\label{e:extrakernelLA}
\{\theta_{1,i,j; \lambda_1 + s_{1,2}} \mid  1\le i \le \nu_1 , \  1 \le j \le \nu_{2} \}.
\end{eqnarray}
\end{itemize} 
\end{enumerate}
\end{Lemma}
\begin{proof}
Part (1) is essentially the same as \cite[Lemma~2.6]{GT19c}. 

Compare \eqref{e:tauonshiftedcurrents} with Lemma~\ref{L:tauaction}(iii) to see that the map is $\tau$-equivariant, hence it is a well-defined map $\cc_n(\sigma)^\tau \onto (\g^e)^\tau$.

We go on to describe the kernel in (2). First of all observe that the elements listed there are $\tau$-fixed elements of the kernel of $\cc_n(\sigma)\onto \g^e$. Now observe that $\g_0^e$ has a spanning set consisting of elements of the form $c_{a,i,b,j}^{(r)} + \tau(c_{a,i,b,j}^{(r)})$. By Lemma~\ref{L:centraliserlemma} we can show that the elements in (2) generate the kernel by checking that the ideal $\mi$ which they generate contains
\begin{align}
\label{e:inthekernel}
\begin{array}{l}
\{\eta_{a,i,j;r} \mid a=1,...,m, \ , 1\le i,j \le \nu_a \ , \ r \ge \lambda_a\} \\ 
\ \ \ \ \ \ \ \ \ \ \ \ \ \cup \{\theta_{a,i,j;r} \mid a=1,..,m-1,  \ 1\le i\le \nu_a,  \, 1\le j \le \nu_{a+1}, \ r \ge s_{a,a+1} + \lambda_a \}.
\end{array}
\end{align}
Consider the case $\epsilon = 1$ in (2)(i) and note that in this case for each $i$ there is a $j$ such that we have $\eta_{1,i,j;\lambda_1} \neq 0$. This is obvious if $\nu_1 \neq 1$ whereas if $\nu_1=1$ we have that $\lambda_1$ is odd by  $-\varepsilon (-1)^{\lambda_1}, \sigma)$ admissibility and hence $\eta_{1,1,1,\lambda_{1}}=d_{1,1,1,\lambda_{1}}+(-1)^{\lambda_1-1}d_{1,1,1,\lambda_{1}} \neq 0$. Using relation \eqref{currel2} we see that $\mi$ contains $\{\eta_{1,i,j;\lambda_1}, \theta_{1,j,k;r}\} = \theta_{1,i,k; \lambda_1 + r}$ for $r \ge s_{1,2}$ and all admissible $i,k$ whilst \eqref{currel3} together with the fact that $\eta_{1,i,j,\lambda_2+r} \in \mi$ for $r \geq 0$ shows that $\eta_{2,i,j;\lambda_2+r} \in \mi$ for $r\ge 0$ and all admissible $i,j$. By Theorem~\ref{dirredcurrent} the subalgebra of $\cc_n(\sigma)^\tau$ generated by $\eta_{a,i,j;r}, \theta_{a,i,j;r}$ with $2\le a$ is isomorphic to a current Lie algebra associated with a lower rank general linear Lie algebra and a shape of strictly smaller length. The description of the kernel follows by induction.

In the second case we see that if $\nu_1 \neq 1$  for each $i$ there is some $j$ such that $\eta_{1,i,j;\lambda_1}\neq 0$. Hence the ideal $\mi$ contains $\{\eta_{1,i,j;\lambda_1}, \theta_{1,j,k; s}\} = \theta_{1,i,k;\lambda_1+s}$ for $s \ge s_{1,2}$, proving (ii).

If $\nu_1=1$ we see that  $\eta_{1,1,1;\lambda_1+1} \neq 0$. Hence the ideal $\mi$ contains $\{\eta_{1,1,1;\lambda_1+1}, \theta_{1,1,k; s}\} = \theta_{1,1,k;\lambda_1 + 1 + s}$ for $s \ge s_{1,2}$.
Together with the additional generators $\theta_{1,i,j;\lambda_1 + s_{1,2}}$ this gives all elements of the form $\theta_{1,i,j;r}$ listed in \eqref{e:inthekernel}. 

Furthermore if $\nu_1=1$ we have by  $(-\ve(-1)^{\lambda_1}, \sigma)$ admissibility that $\lambda_1$ is even. In particular by \eqref{currel7} we have that $ \{\theta_{1,i,j;s_{1,2}+\lambda_1}, \{\theta_{1,i',k;s_{1,2}}, \theta_{2,k,g;s_{2,3}}\}\} +\{\theta_{1,i,j;s_{1,2}}, \{\theta_{1,i',k;s_{1,2}+\lambda_1}, \theta_{2,k,g;s_{2,3}}\}=\frac{1}{2}(\beta)^{i+j}(-1)^{r+s_{a,a+1}}\theta_{2,j',g;\lambda_2+s_{2,3}}$ as we have $\lambda_2=\lambda_1+2s_{1,2}$. So all additional generators of the form $\theta_{2,i,j;\lambda_2 + s_{2,3}}$ also lie in the ideal. 
Now the inductive proof of (iii) proceeds in the same way as in the first case.
\end{proof}

\subsection{The semiclassical shifted Yangian}

In this section we fix $n > 0$, an $n\times n$ shift matrix $\sigma$ and a $\sigma$-admissible shape $\nu$. The {\em (semiclassical) shifted Yangian with respect to an admissible shape} $^{\nu}y_n(\sigma)$ is the Poisson algebra generated by the set
\begin{equation}
\begin{split}
&\{d_{a;i,j}^{(r)} | {1 \leq a \leq m, 1 \leq i,j \leq \nu_a,
r > 0}\},\\
& \hspace{40pt} \cup \{e_{a;i,j}^{(r)} | {1 \leq a < m, 1 \leq i 
\leq \nu_a, 1 \leq j \leq \nu_{a+1},
r > s_{a,a+1}(\nu)}\},\\
&\hspace{80pt} \{f_{a;i,j}^{(r)} \mid {1 \leq a < m, 1 \leq i 
\leq \nu_{a+1}, 1 \leq j \leq \nu_{a},
r > s_{a+1,a}(\nu)}\} 
\end{split}
\end{equation}
together with relations
\begin{align}\label{yangrel1}
\{d_{a;i,j}^{(r)}, d_{b;h,k}^{(s)}\} &=
\delta_{a,b}
\sum_{t=0}^{\min(r,s)-1}
\left(
d_{a;i,k}^{(r+s-1-t)}d_{a;h,j}^{(t)} 
-d_{a;i,k}^{(t)}d_{a;h,j}^{(r+s-1-t)} \right),
\end{align}
\begin{align}\label{yangrel2}
\{e_{a;i,j}^{(r)}, f_{b;h,k}^{(s)}\}
&=
-\delta_{a,b} \sum_{t=0}^{r+s-1}
\widetilde{d}_{a;i,k}^{(r+s-1-t)}d_{a+1;h,j}^{(t)},
\end{align}
\begin{align}\label{yangrel3}
\{d_{a;i,j}^{(r)}, e_{b;h,k}^{(s)}\} &=
\delta_{a,b} 
\sum_{t=0}^{r-1} 
\sum_{g=1}^{\nu_a}d_{a;i,g}^{(t)} e_{a;g,k}^{(r+s-1-t)}\delta_{h,j}
- \delta_{a,b+1} \sum_{t=0}^{r-1}
d_{b+1;i,k}^{(t)} e_{b;h,j}^{(r+s-1-t)},
\end{align}
\begin{align}\label{yangrel4}
\{d_{a;i,j}^{(r)}, f_{b;h,k}^{(s)}\} &=
\delta_{a,b+1} \sum_{t=0}^{r-1}
 f_{b;i,k}^{(r+s-1-t)}d_{b+1;h,j}^{(t)}-\delta_{a,b} 
\delta_{i,k}\sum_{t=0}^{r-1}\sum_{g=1}^{\nu_a} 
f_{a;h,g}^{(r+s-1-t)}d_{a;g,j}^{(t)},
\end{align}
\begin{align} \label{yangrel5}
\{e_{a;i,j}^{(r)}, e_{a;h,k}^{(s)} \}
&=
\sum_{t=r}^{s-1} e_{a;i,k}^{(t)} e_{a;h,j}^{(r+s-1-t)} \ \ \ \   \text{ if } r<s 
\end{align}
\begin{align}\label{yangrel6}
\{f_{a;i,j}^{(r)}, f_{a;h,k}^{(s)} \}
&=
\sum_{t=s}^{r-1} f_{a;i,k}^{(r+s-1-t)}f_{a;h,j}^{(t)}  \ \ \ \   \text{ if } s<r \end{align}
\begin{align}\label{yangrel7}
\{e_{a;i,j}^{(r)}, e_{a+1;h,k}^{(s+1)}\}
-\{e_{a;i,j}^{(r+1)}, e_{a+1;h,k}^{(s)}\}
&=
-\sum_{g=1}^{\nu_{a+1}}e_{a;i,g}^{(r)} e_{a+1;g,k}^{(s)}\delta_{h,j},
\end{align}
\begin{align}\label{yangrel8}
\{f_{a;i,j}^{(r+1)}, f_{a+1;h,k}^{(s)}\}
-\{f_{a;i,j}^{(r)}, f_{a+1;h,k}^{(s+1)}\}
&=
-\delta_{i,k}\sum_{g=1}^{\nu_{a+1}}f_{a+1;h,g}^{(s)}f_{a;g,j}^{(r)},
\end{align}
\begin{align} \label{yanrel9}
\{e_{a;i,j}^{(r)}, e_{b;h,k}^{(s)}\} &= 0\:\:\qquad\text{if $b>a+1$ 
or if $b = a+1$ and $h \neq j$},
\end{align}
\begin{align}\label{yangrel10}
\{f_{a;i,j}^{(r)}, f_{b;h,k}^{(s)}\}] &= 0\:\:\qquad\text{if $b > a+1$
or if $b=a+1$ and $i \neq k$},
\end{align}
\begin{align} \label{yangrel11}
\Big\{e_{a;i,j}^{(r)}, \{e_{a;h,k}^{(s)}, e_{b;f,g}^{(t)}\}\Big\}
+
\Big\{e_{a;i,j}^{(s)}, \{e_{a;h,k}^{(r)}, e_{b;f,g}^{(t)}\}\Big\} &= 0
\qquad\text{if }|a-b|=1,
\end{align}
\begin{align}\label{yangrel12}
\Big\{f_{a;i,j}^{(r)}, \{f_{a;h,k}^{(s)}, f_{b;f,g}^{(t)}\}\Big\}
+
\Big\{f_{a;i,j}^{(s)}, \{f_{a;h,k}^{(r)}, f_{b;f,g}^{(t)}\}\Big\} &= 0
\qquad\text{if }|a-b|=1,
\end{align}
for all admissible $a,b,f,g,h,i,j,k,r,s,t$.
In these relations we use the notation $d_{a,i,j}^{(0)}=\widetilde{d}_{a,i,j}^{(0)}=\delta_{i,j}$  and define $\widetilde{d}_{a,i,j}^{(r)}$ recursively. 
\begin{eqnarray}
\label{e:dtildedefinition}
\widetilde{d}_{a,i,j}^{(r)} := -\sum_{g=1}^{\nu_a}\sum_{t=1}^r d_{a,i,g}^{(t)} \widetilde d_{a,g,j}^{(r-t)}
\end{eqnarray}

In order to describe the structure of $^{\nu}y_n(\sigma)$ as a commutative algebra we make the notation
\begin{eqnarray}
e_{a,i,a+1,j}^{(r)} := e_{a,i,j}^{(r)} \text{ for } 1\le a < m, \ 1 \leq i \leq \nu_{a}, \ 1 \leq j \leq \nu_{a+1} , \ s(\nu)_{a,a+1} < r ;\\ [3pt]
f_{a+1,i,a,j}^{(r)} := f_{a,i,j}^{(r)} \text{ for } 1\le i <m, \ 1 \leq i \leq \nu_{a+1}, \ 1 \leq j \leq \nu_{a} , \ s(\nu)_{a+1,a} < r ,
\end{eqnarray}
and inductively define
\begin{eqnarray}
\label{e:eijrels}
\begin{array}{l}
  e_{a,i,b,j}^{(r)} := \{e_{a,i,b-1,k}^{(r-s_{b-1,b}+1)}, e_{b-1,k,j}^{(s_{b-1,b}+1)}\} \vspace{6pt}\\ \hspace{80pt} \text{ for } 1\le a < b\le m,  \ 1 \leq i \leq \nu_{a}, \ 1 \leq j \leq \nu_{b} , \ s_{a,b} < r ;
  \end{array}\vspace{6pt}\\
  \begin{array}{l}
\label{e:fijrels}
 f_{a,i,b,j}^{(r)} := \{f_{b-1,i,k}^{(s_{b,b-1} + 1)}, f_{a,k,b-1,j}^{(r-s_{b,b-1})}\}  \vspace{6pt}\\ \hspace{80pt} \text{ for } 1\le a < b\le m,  \ 1 \leq i \leq \nu_{a}, \ 1 \leq j \leq \nu_{b} , \ s_{b,a} < r .
 \end{array}
\end{eqnarray}
For $1 \leq k \leq \nu_{b-1}$ note that by the relations this is independent of the choice of $k$ see \cite[3.16]{BK06}.
The shifted Yangian admits a Poisson grading $^{\nu}y_n(\sigma) = \bigoplus_{r \ge 0} {}^{\nu}y_n(\sigma)_r$, which we call the {\it canonical grading}. It places $d_{a,i,j}^{(r)}, e_{a,i,j}^{(r)}, f_{a,i,j}^{(r)}$ in degree $r$ and the bracket lies in degree $-1$, meaning
$\{\cdot, \cdot\} : \ ^{\nu}y_n(\sigma)_r \times \ ^{\nu}y_n(\sigma)_s \to \ ^{\nu}y_n(\sigma)_{r+s-1}$.
There is also an important Poisson filtration $^{\nu}y_n(\sigma) = \bigcup_{i \ge 0} \F_i \ ^{\nu}y_n(\sigma)$, called the {\it loop filtration}, defined by placing $d_{a,i,j}^{(r)}, e_{a,i,j}^{(r)}, f_{a,i,j}^{(r)}$ in degree $r-1$. With respect to this filtration, the bracket is in degree zero  so that $\F_r y_n(\sigma) \times \F_s y_n(\sigma) \to \F_{r+s}y_n(\sigma)$. The associated graded Poisson algebra is denoted $\gr y_n(\sigma)$.

The following theorem is a semiclassical analogue of \cite[Theorem~2.1]{BK06}, which we ultimately deduce from the non-commutative setting, analogous to \cite[Theorem~3.4]{To23}.
\begin{Theorem}
\label{T:shiftedyangianPBW}
Let $\sigma$ be a symmetric shift matrix and $\nu$ a $\sigma$-admissible shape. There is a Poisson isomorphism $S(\cc_n(\sigma)) \isoto \gr \  ^{\nu}y_n(\sigma)$ defined by
\begin{eqnarray}
\label{e:loopiso}
\begin{array}{rcl}
e_{a,i,j;r-1} \longmapsto e_{a,i,j}^{(r)} + \F_{r-2}y_n(\sigma) & \text{ for } & 1\le i<n, \ s_{i,i+1} < r;\\
f_{a,i,j;r-1} \longmapsto f_{a,i,j}^{(r)} + \F_{r-2}y_n(\sigma)& \text{ for } &1\le i<n, \ s_{i+1,i} < r;\\
d_{a,i,j;r-1} \longmapsto d_{a,i,j}^{(r)} + \F_{r-2}y_n(\sigma) & \text{ for } & 1\le i \le n, \ r<0.
\end{array}
\end{eqnarray}
As a consequence $y_n(\sigma)$ is isomorphic to the polynomial algebra on infinitely many variables
\begin{equation}\label{e:Ygens}
\begin{split}
 &\{d_{a,i,j}^{(r)} \mid 1\leq a \leq m, \ 1 \le i,j \le \nu_{a} \ 0 < r\} \\
 & \hspace{40pt}\cup \{e_{a,i,b,j}^{(r)} \mid 1\leq a < b\le m, \ 1 \le i \le \nu_{a}, \ 1 \le j \le \nu_{b} \ s_{a,b} < r \} \\ 
& \hspace{80pt} \cup \, \{f_{a,i,b,j}^{(r)} \mid 1\leq b < a \le m, \ 1 \le i \le \nu_{a}, \ 1 \le j \le \nu_{b} \ s_{b,a} < r \}.
\end{split}
\end{equation}
\end{Theorem}
\begin{proof}
Comparing  the top graded components of the relations \eqref{yangrel1}--\eqref{yangrel12} with respect to the loop filtration, with relations \eqref{e:shiftedcurrentfreerels1}--\eqref{e:shiftedcurrentfreerels9}, it is straightforward to see that \eqref{e:loopiso} gives a surjective Poisson homomorphism. To prove the theorem we demonstrate that the ordered monomials in the elements \eqref{e:Ygens} are linearly independent.

Consider the set
\begin{equation}
\begin{array}{rcl}
X &:= & \{E_{a,i,j}^{(r)} \mid 1 \le a < m, \ 1 \le i \le \nu_a, \ 1 \le j \le \nu_{a+1},  \ s_{a,a+1} < r\} \\
& & \hspace{40pt}\cup \{F_{a,i,j}^{(r)} \mid 1 \le a < m, \ 1 \le i \le \nu_{a+1}, \ 1 \le j \le \nu_{a},  \ s_{a+1,a} < r\} \\
& & \hspace{80pt}\cup \{D_{a,i,j}^{(r)} \mid 1 \le a < m, \ 1 \le i,j \le \nu_a,  \ 0 < r\}.
\end{array}
\end{equation}
In  \cite[\textsection 3]{BK06} the shifted Yangian $^{\nu}Y_n(\sigma)$ is defined as a quotient of the free algebra $\C\langle X \rangle$ by the ideal generated by the relations \cite[(3.3)--(3.14)]{BK06}. Let $L := L_X$ be the free Lie algebra on $X$ and define a grading $L = \bigoplus_{i \ge 0} L_i$ by placing $E_{a,i,j}^{(r)}, F_{a,i,j}^{(r)}, D_{a,i,j}^{(r)}$ in degree $r-1$. Then we place a filtration on the enveloping algebra $U(L)$ so that $L_i$ lies in degree $i+1$. By the PBW theorem for $U(L)$ we see that $\gr U(L) \cong S(L)$.
 
The universal property of $U(L)$ ensures that there is a surjective algebra homomorphism $U(L) \onto \C\langle X \rangle$ and by \cite[I, Ch. IV, Theorem 4.2]{Ser06} this is an isomorphism. Identifying these algebras, the filtration on $U(L)$ descends to $^{\nu}Y_n(\sigma) = \bigcup_{i\ge 0} \F'_i \, {}^{\nu}Y_n(\sigma)$, and this resulting filtration is commonly referred to as the canonical filtration \cite[\textsection 5]{BK06}. The associated graded algebra $\gr  {}^{\nu}Y_n(\sigma)$ is equipped with a Poisson structure in the usual manner. Comparing relations \eqref{yangrel1}--\eqref{yangrel12} with the top graded components of relations \cite[(3.3)--(3.14)]{BK06} we see that the Poisson surjection $S(L) \onto \gr {}^{\nu}Y_n(\sigma)$ factors through $S(L) \to {}^{\nu}y_n(\sigma)$. As a result there is a surjective Poisson homomorphism $\pi: {}^{\nu}y_n(\sigma) \onto \gr  {}^{\nu}Y_n(\sigma)$ given by $e_{a,i,j}^{(r)} \mapsto E_{a,i,j}^{(r)} + \F'_{r-1} {}^{\nu}Y_n(\sigma), f_{a,i,j}^{(r)} \mapsto F_{a,i,j}^{(r)} + \F'_{r-1}  {}^{\nu}Y_n(\sigma), d_{a,i,j}^{(r)} \mapsto D_{a,i,j}^{(r)} + \F'_{r-1} {}^{\nu}Y_n(\sigma)$. Following \cite[(3.15), (3.16)]{BK06} we introduce elements $E_{a,i,b,j}^{(r)}, F_{a,i,b,j}^{(r)}$ of $^{\nu}Y_n(\sigma)$ lying in filtered degree $r$. By the definition of the filtration and the elements \eqref{e:eijrels}, \eqref{e:fijrels} we have
 \begin{eqnarray*}
 \label{e:piandthegenerators}
 \begin{array}{rcl}
 \pi(e_{a,i,b,j}^{(r)}) & = & E_{a,i,b,j}^{(r)} + \F'_{r-1} \ ^{\nu}Y_n(\sigma);\vspace{6pt}\\
 \pi(f_{a,i,b,j}^{(r)}) & = & F_{a,i,b,j}^{(r)}  + \F'_{r-1} \ ^{\nu}Y_n(\sigma).\\
 \end{array}
 \end{eqnarray*}
By \cite[Theorem~3.2(iv)]{BK06} the ordered monomials in $E_{a,i,b,j}^{(r)}, F_{a,i,b,j}^{(r)}, D_{a,i,j}^{(r)}$
are linearly independent in $^{\nu}Y_n(\sigma)$ and so we deduce that the images of these monomials in $\gr ^{\nu}Y_n(\sigma)$ are linearly independent. This completes the proof.
\end{proof}
We record two formulas for future use.
\begin{Lemma}\label{tildelemma}
The following hold for $i=1,...,n$, $j = 1,...,n-1$, $r > 0$ and $s > s_{j,j+1}$:
\begin{eqnarray}
\label{e:tildedone}
\{\td_{a,i,j}^{(r)}, e_{b,h,k}^{(s)}\} =\delta_{a,b+1}\delta_{i,k}\sum_{t=0}^{r-1}\sum_{g=1}^{\nu_a} 
e_{b;h,g}^{(r+s-1-t)}\td_{a;g,j}^{(t)}-\delta_{a,b} 
 \sum_{t=0}^{r-1}
 e_{a;i,k}^{(r+s-1-t)}\td_{a;h,j}^{(t)},\\
\label{e:tildedonf}
\{\td_{a,i,j}^{(r)}, f_{b,h,k}^{(s)}\big\} = 
\delta_{a,b} \sum_{t=0}^{r-1}
\td_{a;i,k}^{(t)} f_{a;h,j}^{(r+s-1-t)}
- \delta_{a,b+1} \sum_{t=0}^{r-1} 
\sum_{g=1}^{\nu_a}\td_{a;i,g}^{(t)} f_{b;g,k}^{(r+s-1-t)}\delta_{h,j}.
\end{eqnarray}
\end{Lemma}
\begin{proof}
We only sketch \eqref{e:tildedone}, as the proof of \eqref{e:tildedonf} is almost identical. The argument is by induction based on \eqref{e:dtildedefinition}. Note that $\td_{a,i,j}^{(1)} = -d_{a,i,j}^{(1)}, \td^{(0)}_{a,i,j}=d^{(0)}_{a,i,j}=\delta_{i,j}$ and so \eqref{e:tildedone} is equivalent to \eqref{yangrel3} in this case. We have $\{\td_{a,i,j}^{(r)}, e_{b,h,k}^{(s)}\} = \{-\sum_{g=1}^{\nu_a}\sum_{t=1}^{r} d_{a,i,g}^{(t)} \td_{a,g,j}^{(r-t)}, e_{b,h,k}^{(s)}\}$ and relation \eqref{yangrel3} together with the inductive hypothesis imply that the coefficient of $e_{b,m,n}^{(s+t)}$ is 
\begin{align*}
& \delta_{a,b+1}\delta_{m,h}(\sum_{w=t+1}^{r-1}\td_{a,n,j}^{(r-w)} d_{a,i,k}^{(w-t-1)}-\td_{a,n,j}^{(w-t-1)} d_{a,i,k}^{(r-w)}+d_{a,i,k}^{(r-t-1)}\td_{a,n,j}^{(0)})  \\
&\hspace{20pt}-(\delta_{a,b}\delta_{n,k}(\sum_{w=t+1}^{r-1}\td_{a,h,j}^{(r-w)} d_{a,i,m}^{(w-t-1)}-\td_{a,h,j}^{(w-t-1)} d_{a,i,m}^{(r-w)}+d_{a,i,m}^{(r-t-1)}\td_{a,h,j}^{(0)})).
\end{align*}
Using $d_{a,i,j}^{(0)} = \td_{a,i,j}^{(0)} = \delta_{i,j}$ we see that all quadratic terms cancel, whilst the linear term equates to $\delta_{a,b+1}\delta_{m,h}\td_{a,n,j}^{(r-t-1)}\delta_{i,k} - \delta_{a,b}\delta_{n,k} \delta_{i,m}\td_{a,h,j}^{(r-t-1)}$, which concludes the induction.
\end{proof}

\subsection{The Dirac reduction of the shifted Yangian}
\label{ss:PDreductionofyangians}
In this section we suppose that $\sigma$ is symmetric and we let $\nu=(\nu_1,\dots, \nu_m)$ be a $(\beta,\sigma)$-admissible shape. Examining the relations \eqref{yangrel1}--\eqref{yangrel12} we see that there is unique involutive Poisson automorphism $\tau$ of $y_n(\sigma)$ determined by
\begin{eqnarray}
\label{e:tauon}
\begin{array}{rcl}
\tau(d_{a,i,j}^{(r)})& := & (\beta)^{i+j}(-1)^{r} d_{a,(\nu)_a + 1 - j,(\nu)_a + 1 - i}^{(r)} = (\beta)^{i+j}(-1)^{r} d_{a,j',i'}^{(r)} \\ [3pt]
\tau(e_{a,i,j}^{(r)}) & := & (\beta)^{i+j}(-1)^{r+ s_{a,a+1}} f_{a,\nu_{a+1}+1-j,\nu_a+1-i}^{(r)}=(\beta)^{i+j}(-1)^{r+ s_{a,a+1}} f_{a,j',i'}^{(r)} \\ [3pt]
\tau(f_{a,i,j}^{(r)}) & := & (\beta)^{i+j}(-1)^{r + s_{a+1, a}} e_{a,\nu_{a+1}+1-j,\nu_{a}+1-i}^{(r)} =  (\beta)^{i+j}(-1)^{r + s_{a+1, a}} e_{a,j',i'}^{(r)} .
\end{array}
\end{eqnarray}
\begin{Remark}
Note that for this automorphism to be involutive we need $\nu$ to be $(\beta,\sigma)$-admissible and that the automorphism $\tau$ depends on the shape.
\end{Remark}
Our present goal is to give a complete description of the Dirac reduction $R(^{\nu}y_n(\sigma), \tau)$. Before doing so we will state the relationship between $R(^{\nu}y_n(\sigma), \tau)$, $R(^{\mu}y_n(\sigma), \tau)$ where both $\mu$ and $\nu$ are $(\beta,\sigma)$-admissble.
\begin{Proposition}\label{equivariant iso}
Let $\nu$, $\mu$ be $(\beta, \sigma)$-admissible shapes. Denote by ${}^{\nu}\tau$, ${}^{\mu}\tau$ the involutions on $^{\nu}y_n(\sigma)$, $^{\mu}y_n(\sigma)$ as defined in \eqref{e:tauon}. Then there exists a $\mathbb{Z}/2$-equivariant isomorphism $$^{\nu}y_n(\sigma) \isoto {}^{\mu}y_n(\sigma)$$ where the $\mathbb{Z}/2$-action on the left is given by ${}^\nu \tau$ and on the right by ${}^\mu \tau$.
\end{Proposition}
\begin{proof}
We will prove the result for  $GL_{\mu_1} \times \dots \times GL_{\mu_m} \subseteq GL_{\nu_1} \times \dots \times GL_{\nu_{m'}}$. This suffices as for $\beta=1$ the shape $(1,1,\dots,1,1)$ is always $(1,\sigma)$-admissible and if $\sigma$ admits a  $\beta=-1$ admissible shape then the shape  $(2,2,\dots,2,2)$ is always $(-1,\sigma)$-admissible. The general result then follows by composing isomorphisms.
Let ${}^{\nu}Y_n(\sigma)$ denote the shifted Yangian of shape $\nu$ as defined in \cite[§3]{BK06}. Also write $Y_n$ for the Yangian for $\gl_n$. Recall the definition of parabolic presentations of $Y_n$ from \cite[§6]{BK05} where the elements $E_{a,,b,i,j}^{(r)},F_{a,b,i,j}^{(r)},D_{a,i,j}^{(r)}$ are defined in terms of a Gauss decomposition of the $T$-matrix or more explicitly by quasi determinants \cite[(6.2)-(6.4)]{BK05}. By \cite[Corollary 3.3]{BK06} the canonical map $i_{\nu}:{}^{\nu}Y_n(\sigma) \to Y_n$ given by sending the generators to the elements with the same name is injective and its image is independent of the choice of $\nu$. Furthermore by \cite[§5]{BK05} the map is filtered with respect to the canonical filtration.
Consider now the map $\tau_{\nu}:Y_n \to Y_n: T(u) \mapsto J_{\nu}^{-1}T(-u)^tJ_{\nu}$ where 
\begin{equation}
\label{Jdefn}
J_{\nu}:=  \sum_{a=1}^{m} \sum_{i=1}^{\nu_a} (\beta)^{i}(-1)^{s_{1,a}(\nu)}  e_{a,i;a,\nu_a + 1 - i} \in GL_{\nu_1} \times \dots \times GL_{\nu_m}
\end{equation}

By \cite[§1.3]{Mo07} this is an automorphism. 
Note that $J_{\nu}$ is a block diagonal matrix that can be described as follows. For each index $a = 1,...,m$ there is a block of size $\nu_a$, each of these blocks has entries $(-1)^{s_{1,a}(\nu)}$ on the antidiagonal  and zeros elsewhere when $\beta=1$. When $\beta=-1$ it has alternating entries  $(-1)^{s_{1,a}(\nu)}(\mp 1)$ on the antidiagonal and zeroes elsewhere. Thus when $\beta = 1$ the matrix $J$ is a block diagonal matrix where each block is symmetric and when $\beta = -1$ each block is anti-symmetric. It follows that 
\begin{eqnarray}
\label{e:Jinverse}
J_{\nu}^{-1} = \beta J_{\nu}.
\end{eqnarray} \\
We claim that
\begin{eqnarray}
\tau_{\nu}(E_{a,a+1,i,j}^{(r)})&= & (\beta)^{i+j}(-1)^{r+s_{a,a+1}(\nu)}F_{a+1,a,j',i'}^{(r)} \\
\tau_{\nu}(F_{a+1,a,i,j}^{(r)})&=&(\beta)^{i+j}(-1)^{r+s_{a,a+1}(\nu)}E_{a+1,a,j',i'}^{(r)} \\
\tau_{\nu}(D_{a,a,i,j}^{(r)})&=&(\beta)^{i+j}(-1)^{r}D_{a,a,j',i'}^{(r)}
\end{eqnarray}
and in particular the image of $i_{\nu}$ is stabilized by $\tau_{\nu}$.
Indeed by \cite[(6.6)-(6.8)]{BK05} we have that for the automorphism 
$\omega:Y_n \to Y_n: T(u) \mapsto T(-u)^\top$
\begin{equation} 
\begin{array}{rcl}
\omega(E_{a,a+1,i,j}^{(r)}) &=& (-1)^{r}F_{a+1,a,j,i}^{(r)} \\
\omega(F_{a+1,a,i,j}^{(r)})&=&(-1)^r E_{a,a+1,j,i}^{(r)} \\
\omega(D_{a,a,i,j}^{(r)})&=&(-1)^r D_{a,a,j,i}^{(r)}
\end{array}
\end{equation}
 Furthermore using the explicit description of $E_{a,a+1,i,j}^{(r)}$, $F_{a+1,a,i,j}^{(r)}$, $D_{a,a,i,j}^{(r)}$ in terms of quasi determinants and \cite[§1.3(i),(ii)]{GGRW04} it  follows immediately that for any matrix $A=(A_{\nu_1}, \dots, A_{\nu_m}) \in GL_{\nu_1} \times \cdots \times GL_{\nu_m}$ and for the automorphism given by $c_A:T(u) \mapsto A^{-1}T(u)A$  we have
\begin{equation} 
\label{conjongens}
\begin{array}{rcl}
c_A(E_{a,a+1}^{(r)}) &= & A_{\nu_a}^{-1}E_{a,a+1}^{(r)}A_{\nu_{a+1}} \\
c_A(F_{a+1,a}^{(r)})&=& A_{\nu_{a+1}}^{-1}F_{a+1,a,}^{(r)}A_{\nu_a} \\
c_A(D_{a,a}^{(r)})&=& A_{\nu_a}^{-1}D_{a,a}^{(r)}A_{\nu_a}
\end{array}
\end{equation}
where $(E_{a,a+1}^{(r)})_{i,j}=E_{a,a+1,i,j}^{(r)}$, $(F_{a+1,a}^{(r)})_{i,j}=F_{a+1,a,i,j}^{(r)}$ $(D_{a,a}^{(r)})_{i,j}=D_{a,a,i,j}^{(r)}$.
 The claim then follows by direct calculation.
 
 From the description of the matrices $J_{\nu}$ and $J_{\mu}$ it follows that there exists $B \in GL_{\nu_1} \times \dots \times GL_{\nu_m}$ such that $J_{\mu}=B^{t}J_{\nu}B$.  It follows from \eqref{conjongens} that the automorphism
$c_{B}:Y_n \to Y_n : T(u) \mapsto B^{-1}T(u)B$ stabilizes the image of $i_{\nu}$ and that it is filtered with respect to the canonical filtration as defined in \cite[§5]{BK06}. It thus follows that $i_{\mu}^{-1} \circ c_{B} \circ i_{\nu}:{}^{\nu}Y_n(\sigma) \to {}^{\mu}Y_n(\sigma) $ is well defined and filtered.

Recall from the proof of Theorem \ref{T:shiftedyangianPBW} that the map  $\pi_{\nu}: {}^{\nu}y_n(\sigma) \onto \gr \ ^{\nu}Y_n(\sigma)$ given by $e_{a,i,j}^{(r)} \mapsto E_{a,i,j}^{(r)} + \F'_{r-1} \ ^{\nu}Y_n(\sigma), f_{a,i,j}^{(r)} \mapsto F_{a,i,j}^{(r)} + \F'_{r-1}  \ ^{\nu}Y_n(\sigma), d_{a,i,j}^{(r)} \mapsto D_{a,i,j}^{(r)} + \F'_{r-1} \  ^{\nu}Y_n(\sigma)$  is an isomorphism of Poisson algebras, where $\F'$ denotes the canonical filtration.  In particular $\pi_{\mu}^{-1} \circ \gr( i_{\mu}^{-1} \circ c_{B} \circ i_{\nu}) \circ \pi_{\nu}$
 is the desired isomorphism.
\end{proof}
\begin{Corollary}\label{Dirrediso}
Let $\nu$, $\mu$ be $(\beta, \sigma)$-admissible shapes.
Then the Dirac reductions $R({}^{\nu}y_n(\sigma), {}^{\nu}\tau)$ and $R({}^{\mu}y_n(\sigma), {}^{\mu}\tau)$ are Poisson isomorphic. $\hfill \qed$
\end{Corollary}

\begin{Remark}\label{Proof of thm B}
The proof of Theorem B now follows by combining Corollary \ref{Dirrediso} with Theorem \ref{yangianquantizes} which is the special case for $\mu=(1,1,\dots,1,1)$, stating that $y_n^+(\sigma) \cong R({}^{\mu}y_n(\sigma),\tau)$. The parabolic presentations of $y_n^+(\sigma)$ alluded to in Theorem~B are actually the parabolic presentations $R(^{\nu}y_n(\sigma),\tau))$ which are formulated in Theorem~\ref{dirredyangpres}.
\end{Remark}

For all admissible $a,i,j,r$ we write
\begin{eqnarray}
\label{e:ehatandcheck}
\begin{array}{rcl}
\hd_{a,i,j}^{(r)} & := & \frac{1}{2}(d_{a,i,j}^{(r)} + (\beta)^{i+j}(-1)^{r} d_{a,j',i'}^{(r)}) \in  \ ^{\nu}y_n(\sigma); \vspace{4pt}\\
\cd_{a,i,j}^{(r)} & := & \frac{1}{2}(d_{a,i,j}^{(r)} - (\beta)^{i+j}(-1)^{r} d_{a,j',i'}^{(r)}) \in  \ ^{\nu}y_n(\sigma);
\vspace{4pt}\\
\he_{a,i,j}^{(r)} & := & \frac{1}{2}(e_{a,i,j}^{(r)} + (\beta)^{i+j}(-1)^{r + s_{a,a+1}} f_{a,j',i'}^{(r)}) \in \ ^{\nu}y_n(\sigma); \vspace{4pt}\\
\ce_{a,i,j}^{(r)} & := & \frac{1}{2}(e_{a,i,j}^{(r)} - (\beta)^{i+j}(-1)^{r + s_{a,a+1}} f_{a,j',i'}^{(r)}) \in  \ ^{\nu}y_n(\sigma),
\end{array}
\end{eqnarray}
\begin{Remark}\label{scaling}
Compare \eqref{e:ehatandcheck} with \cite[(3.59)]{To23} to see that in the case $\nu=(1,1,\dots,1,1)$ we have $\hd_{i,1,1}^{(r)}=\hd_{i}^{(r)}$ and $\he_{i,1,1}^{(r)}= \frac{1}{2} \he_{i}^{(r)}$.
\end{Remark}
Now  let $\cy_n(\sigma)$ be the ideal of $y_n(\sigma)$ generated by
$$\ce_{a,i,j}^{(r)},\check{d}_{b,h,k}^{(s)}$$ for all admissible $a,i,j,b,h,k,r,s$.
Also write $\cy_n(\sigma)^\tau := \cy_n(\sigma) \cap y_n(\sigma)^\tau$.
By \cite[Lemma~2.2]{To23} we see that $R(^{\nu}y_n(\sigma), \tau) = ^{\nu}y_n(\sigma)^\tau / ^{\nu}\cy_n(\sigma)^\tau$ is Poisson generated by elements
\begin{eqnarray}
\label{e:thetaandeta}
\begin{array}{rcl}
& & \he_{a,i,j}^{(r)} + \cy_n(\sigma)^\tau \text{ for } a=1,...,m-1, \ 1 \leq i \leq \nu_{a}, \ 1 \leq j \leq \nu_{a+1}, \ s_{a,a+1} < r ,\vspace{4pt}\\
& & \hd_{a,i,j}^{(r)} + \cy_n(\sigma)^\tau \text{ for } a=1,...,m, \ 1 \leq i,j \leq \nu_{a}, \ 0 < r ,
\end{array}
\end{eqnarray}
with Poisson brackets induced by the bracket on $y_n(\sigma)$. We will write  $\theta_{a,i,j}^{(r)}$,  $\eta_{a,i,j}^{(r)}$  for $ \he_{a,i,j}^{(r)} + \cy_n(\sigma)^\tau$ and $\hd_{a,i,j}^{(r)} + \cy_n(\sigma)^\tau$ . Furthermore $R(^{\nu}y_n(\sigma), \tau)$ is generated as a commutative algebra by elements\begin{eqnarray}
\label{e:theRgenerators}
\begin{array}{rcl}
& & \theta_{a,i,b,j}^{(r)} \text{ for } 1\leq a < b\leq m, \ 1 \leq i \leq \nu_{a}, \ 1 \leq j \leq \nu_{b}, \ s_{a,a+1} < r ,\vspace{4pt}\\
& & \eta_{a,i,j}^{(r)} \text{ for } a=1,...,m, \ 1 \leq i,j \leq \nu_{a}, \ 0 < r ,
\end{array}
\end{eqnarray}
where $\theta_{a,i,b,j}^{(r)} := e_{a,i,b,j}^{(r)} +(\beta)^{i+j} (-1)^{r+s_{a,b}} f_{a,j',b,i'}^{(r)}$. Using an inductive argument and \eqref{yangrel2} we see the elements $\theta_{a,i,b,j}^{(r)}$ can also be defined via the following recursion for all admissible $a,b,i,j,r$.
\begin{eqnarray}
\label{e:higherthetas}
\begin{array}{rcl}
& & \theta_{a,i,a+1,j}^{(r)}:= \theta_{a,i,j}^{(r)} \\
&& \theta_{a,i,b,j}^{(r)}:= \{ \theta_{a,i,b-1,,k}^{(r-s_{b-1,b})},\theta_{b-1,k,b,j}^{(s_{b-1,b}+1)} \}
\end{array}
\end{eqnarray}
Thanks to Theorem~\ref{T:shiftedyangianPBW} we see that $R(y_n(\sigma), \tau)$ comes equipped with the {\it canonical grading} $R(y_n(\sigma), \tau) = \bigoplus_{i \ge 0} R(y_n(\sigma), \tau)_i$ which places $\eta_{a,i,j}^{(r)}, \theta_{a,i,j}^{(r)}$ in degree $r$ and the Poisson bracket in degree $-1$. Another crucial feature is the {\it loop filtration} $R(y_n(\sigma), \tau) = \bigcup_{i\ge 0} \F_iR(y_n(\sigma), \tau)$ which places  $\eta_{a,i,j}^{(r)}, \theta_{a,i,j}^{(r)}$ in degree $r-1$ and the bracket in degree $0$. Both of these structures are naturally inherited from $y_n(\sigma)$.

The following is our main structural result on the Dirac reduction of the shifted Yangian.
\begin{Theorem}\label{dirredyangpres}
$R(^{\nu}y_n({\sigma}),\tau)$ is Poisson generated by 
\begin{equation}\label{diryanggens}
\begin{split}
&\{\eta_{a;i,j}^{(r)} | {1 \leq a \leq m, 1 \leq i,j \leq \nu_a,
r > 0}\},\\
&\{\theta_{a;i,j}^{(r)} | {1 \leq a < m, 1 \leq i 
\leq \nu_a, 1 \leq j \leq \nu_{a+1},
r > s_{a,a+1}(\nu)}\},
\end{split}
\end{equation}
 with the following Poisson relations\vspace{6pt}
 \begin{equation}\label{diryangrel0}
\begin{split}
\eta_{a;i,j}^{(r)}=(\beta)^{i+j}(-1)^{r}\eta_{a;j',i'}^{(r)}
  \end{split}  
\end{equation}
\begin{equation}\label{diryangrel1}
\begin{split}
 &\hspace{40pt}\{\eta_{a,i,j}^{(r)},\eta_{b,h,k}^{(s)}\} \\
 &=\delta_{a,b} \frac{1}{2}
\sum_{t=0}^{\min(r,s)-1}
\left(
\eta_{a;i,k}^{(r+s-1-t)}\eta_{a;h,j}^{(t)} 
-\eta_{a;i,k}^{(t)}\eta_{a;h,j}^{(r+s-1-t)}+(\beta)^{h+k}(-1)^{s}\left(\eta_{a;i,h'}^{(r+s-1-t)}\eta_{a;k',j}^{(t)} 
-\eta_{a;i,h'}^{(t)}\eta_{a;k',j}^{(r+s-1-t)} \right) \right)
 \end{split}  
\end{equation}
\begin{equation}\label{diryangrel2}
\begin{split}
  \{\eta_{a,i,j}^{(r)}, \theta_{b,h,k}^{(s)}\}&=
   \frac{1}{2} \left( \delta_{h,j} \delta_{a,b} \sum_{t=0}^{r-1} \sum_{g=1}^{\nu_{\alpha}}   \eta_{a,i,g}^{(t)}\theta_{a,g,k}^{(r+s-1-t)}-\delta_{a,b+1} \sum_{t=0}^{r-1}\eta_{b+1,i,k}^{(t)}\theta_{b,h,j}^{(r+s-1-t)}\right)\\
   &+   (\beta)^{i+j}(-1)^{r} \frac{1}{2}\left( \delta_{h,i'} \delta_{a,b} \sum_{t=0}^{r-1} \sum_{g=1}^{\nu_{\alpha}}   \eta_{a,j',g}^{(t)}\theta_{a,g,k}^{(r+s-1-t)}-\delta_{a,b+1} \sum_{t=0}^{r-1}\eta_{b+1,j',k}^{(t)}\theta_{b,h,i'}^{(r+s-1-t)}\right)
 \end{split}  
\end{equation}
\begin{equation}
\begin{split}
 \{\theta_{a,i,j}^{(r)}, \theta_{a,h,k}^{(s)}\} &=
 \frac{1}{2} \left( \sum_{t=r}^{s-1} \theta_{a,i,k}^{(t)} {\theta}_{a,k,j}^{(r+s-1-t)}+(\beta)^{h+k}(-1)^{s_{a,a+1}+s} \sum_{t=0}^{r+s-1}\widetilde{\eta}_{a,i,h'}^{(r+s-1-t)}\eta_{a+1,k',j}^{(t)} \right)\  \textnormal{for} \ r < s \\
 \end{split}  
\end{equation}
\begin{align}
\{\theta_{a,i,j}^{(r)}, \theta_{a+1,h,k}^{(s+1)}\}-\{\theta_{a,i,j}^{(r+1)}, \theta_{a+1,h,k}^{(s)}\}= -\frac{1}{2}\delta_{h,j}\left( \sum_{g=1}^{\nu_{a+1}}\theta^{(r)}_{a,i,g}\theta^{(s)}_{a+1,g,k}\right)
\end{align}
\begin{align}
\{\theta_{a,i,j}^{(r)}, \theta_{b,h,k}^{(s)} \}=0 \textnormal{ for }  b>a+1  \textnormal{ or if } b=a+1 \textnormal{ and } h \neq j
\end{align}
\begin{align}
    \{\theta_{a,i,j}^{(r)}, \{\theta_{a,h,k}^{(s)}, \theta_{b,f,g}^{(t)}\}\} +\{\theta_{a,i,j}^{(s)}, \{\theta_{a,h,k}^{(r)}, \theta_{b,f,g}^{(t)}\}\}=0 \textnormal{ if } \mid a-b \mid =1 \textnormal{ and } r+s \textnormal{ odd}
\end{align}
\vspace{10pt}
\begin{equation}
\label{diryangrel7}
\begin{split}
    &\{\theta_{a,i,j}^{(r)}, \{\theta_{a,h,k}^{(s)}, \theta_{b,f,g}^{(t)}\}\} +\{\theta_{a,i,j}^{(s)}, \{\theta_{a,h,k}^{(r)}, \theta_{b,f,g}^{(t)}\}\}\\
    & \hspace{20pt}=\frac{1}{2}(\beta)^{i+j}(-1)^{r+ s_{a,a+1}+1}\sum_{t'=0}^{r+s-1} \left(\widetilde{\eta}_{a;h,i'}^{(r+s-1-t')}\delta_{a+1,b} \delta_{f,k}
\sum_{t''=0}^{t'-1} 
\sum_{m=1}^{\nu_{a+1}}\eta_{a+1;j',m}^{(t'')} \theta_{a+1;m,g}^{(t'+t-1-t'')} \right)\\
&\hspace{40pt}+\frac{1}{2}(\beta)^{i+j}(-1)^{r + s_{a,a+1}+1}\sum_{t'=0}^{r+s-1} \left(\eta^{(t')}_{a+1,j',k}\delta_{a,b+1}\delta_{h,g}\sum_{t''=0}^{r-t'-1}\sum_{m=1}^{\nu_a} 
\theta_{b;f,m}^{(r-t'+t-1-t'')} \widetilde{\eta}_{a;m,i'}^{(t'')} \right)\\  
&\hspace{60pt}\textnormal{ if } \mid a-b \mid =1 \textnormal{ and } r+s \textnormal{ even}.
\end{split}
\end{equation}
Here we define $\eta_{a,i,j}^{(0)}=\delta_{i,j}$, $\widetilde{\eta}_{a,i,j}^{(0)}=\delta_{i,j}$ and $\widetilde{\eta}_{a,i,j}^{(r)}=-\sum_{g=1}^{\nu_a} \sum_{t=1}^{r} \eta_{a,i,g}^{(t)} \tilde{\eta}_{a,g,j}^{(r-t)}$
\end{Theorem}
\begin{proof}
In this proof we will frequently use the following two basic facts: $(\beta)^{i+j}=(\beta)^{i'+j'}$, which follows from $(\beta, \sigma)$-admissibility, and  $xy+\tau(xy)=\frac{1}{2}(x+\tau(x))(y+\tau(y))$ for any $x,y \in R(y_n(\sigma),\tau)$. 

First note that $\widetilde{\eta}_{a,i,j}^{(r)}=\frac{1}{2}(\tilde{d}_{a,i,j}^{(r)}+\tau(\tilde{d}_{a,i,j}^{(r)}))$. The proof is by induction on $r$ 
\begin{equation}
\begin{split}
\widetilde{\eta}_{a,i,j}^{(r)}&=-\sum_{g=1}^{\nu_a} \sum_{t=1}^{r} \eta_{a,i,g}^{(t)} \tilde{\eta}_{a,g,j}^{(r-t)}\\
&=-\frac{1}{4}\sum_{g=1}^{\nu_a} \sum_{t=1}^{r} (d_{a,i,g}^{(t)}+\tau(d_{a,i,g}^{(t)})) (\tilde{d}_{a,g,j}^{(r-t)}+\tau(\tilde{d}_{a,g,j}^{(r-t)}))\\
&=-\frac{1}{2} \sum_{g=1}^{\nu_a} \sum_{t=1}^{r} d_{a,i,g}^{(t)}\tilde{d}_{a,g,j}^{(r-t)}+\tau(d_{a,i,g}^{(t)})\tau(\tilde{d}_{a,g,j}^{(r-t)})
\end{split}
\end{equation}
Relation \eqref{diryangrel0} follows immediately.
For relation \eqref{diryangrel1} we have
   \begin{align*}
 \{\eta_{a,i,j}^{(r)},\eta_{b,h,k}^{(s)}\}&=\frac{1}{4} \left(\{d_{a,i,j}^{(r)},d_{b,h,k}^{(s)}\}+ (\beta)^{i+j+h+k}(-1)^{r+s}\{d_{a,j',i'}^{(r)},d_{b,k',h'}^{(s)}\} \right)\\
 &+\frac{1}{4}\left(((\beta)^{i+j}-1)^{r}\{d_{a,j',i'}^{(r)},d_{b,h,k}^{(s)}\}+(\beta)^{h+k}(-1)^{s}\{\eta_{a,i,j}^{(r)},\eta_{b,k',h'}^{(s)}\}\right) \\
 &= \frac{1}{4}\delta_{a,b}
 \sum_{t=0}^{\min(r,s)-1}
\left(
d_{a;i,k}^{(r+s-1-t)}d_{a;h,j}^{(t)} 
-d_{a;i,k}^{(t)}d_{a;h,j}^{(r+s-1-t)} \right)\\
&+\frac{1}{4} \delta_{a,b}(\beta)^{i+j+h+k}(-1)^{r+s}\sum_{t=0}^{\min(r,s)-1} \left(
d_{a;j',h'}^{(r+s-1-t)}d_{a;k',i'}^{(t)} 
-d_{a;j',h'}^{(t)}d_{a;k',i'}^{(r+s-1-t)}\right) \\
&+\frac{1}{4}\delta_{a,b}(\beta)^{i+j}(-1)^{r}\sum_{t=0}^{\min(r,s)-1} \left(
d_{a;j',k}^{(r+s-1-t)}d_{a;h,i'}^{(t)} 
-d_{a;j',k}^{(t)}d_{a;h,i'}^{(r+s-1-t)}\right) \\
&+\frac{1}{4} \delta_{a,b}(\beta)^{h+k}(-1)^{s}\sum_{t=0}^{\min(r,s)-1} \left(
d_{a;i,h'}^{(r+s-1-t)}d_{a;k',j}^{(t)} 
-d_{a;i,h'}^{(t)}d_{a;k',j}^{(r+s-1-t)}\right)\\
&= \frac{1}{2}\delta_{a,b}
\sum_{t=0}^{\min(r,s)-1}
\left(
\eta_{a;i,k}^{(r+s-1-t)}\eta_{a;h,j}^{(t)} 
-\eta_{a;i,k}^{(t)}\eta_{a;h,j}^{(r+s-1-t)} \right) \\
&+\frac{1}{2}\delta_{a,b}(\beta)^{h+k}(-1)^{s}\sum_{t=0}^{\min(r,s)-1}\left(\eta_{a;i,h'}^{(r+s-1-t)}\eta_{a;k',j}^{(t)} 
-\eta_{a;i,h'}^{(t)}\eta_{a;k',j}^{(r+s-1-t)} \right)
 \end{align*}
 Using relation \eqref{yangrel1}.
 For relation \eqref{diryangrel2} we have
 \begin{align*}
\{\eta_{a,i,j}^{(r)}, \theta_{b,h,k}^{(s)}\}&=\frac{1}{4}\left(\{d_{a,i,j}^{(r)},e_{b,h,k}^{(s)}\}+(\beta)^{i+j+h+k}(-1)^{r+s+s_{b,b+1}} \{d_{a,j',i'}^{(r)}, f_{b,k',h'}^{(s)}\}\right) \\
&+\frac{1}{4}\left((\beta)^{i+j}(-1)^{r} \{d_{a,j',i'}^{(r)}, e_{b,h,k}^{(s)}\}+(\beta)^{h+k}(-1)^{s+s_{b,b+1}}
\{d_{a,i,j}^{(r)}, f_{b,k',h'}^{(s)}\}\right)\\
&= \frac{1}{4} \left( \delta_{a,b} 
\sum_{t=0}^{r-1} 
\sum_{g=1}^{\nu_a}d_{a;i,g}^{(t)} e_{a;g,k}^{(r+s-1-t)}\delta_{h,j}
- \delta_{a,b+1} \sum_{t=0}^{r-1}
d_{b+1;i,k}^{(t)} e_{b;h,j}^{(r+s-1-t)} \right) \\
&+\frac{1}{4}(\beta)^{i+j+h+k}(-1)^{r+s+s_{b,b+1}}\left( \delta_{a,b+1} \sum_{t=0}^{r-1}
d_{b+1;k',i'}^{(t)} f_{b;j',h'}^{(r+s-1-t)}- \delta_{a,b} 
\sum_{t=0}^{r-1} 
\sum_{g=1}^{\nu_a}d_{a;g,i'}^{(t)} f_{a;k',g}^{(r+s-1-t)}\delta_{h',j'}
\right)\\
&+ \frac{1}{4}(\beta)^{i+j}(-1)^{r}\left( \delta_{a,b} 
\sum_{t=0}^{r-1} 
\sum_{g=1}^{\nu_a}d_{a;j',g}^{(t)} e_{a;g,k}^{(r+s-1-t)}\delta_{h,i'}
- \delta_{a,b+1} \sum_{t=0}^{r-1}
d_{b+1;j',k}^{(t)} e_{b;h,i'}^{(r+s-1-t)}\right)\\
&+\frac{1}{4}(\beta)^{k+h}(-1)^{s+s_{b,b+1}}\left( \delta_{a,b+1} \sum_{t=0}^{r-1}
d_{b+1;k',j}^{(t)} f_{b;i,h'}^{(r+s-1-t)}- \delta_{a,b} 
\sum_{t=0}^{r-1} 
\sum_{g=1}^{\nu_a}d_{a;g,j}^{(t)} f_{a;k',g}^{(r+s-1-t)}\delta_{h',i}
\right)\\
&=
   \frac{1}{2} \left( \delta_{h,j} \delta_{a,b} \sum_{t=0}^{r-1} \sum_{g=1}^{\nu_{\alpha}}   \eta_{a,i,g}^{(t)}\theta_{a,g,k}^{(r+s-1-t)}-\delta_{a,b+1} \sum_{t=0}^{r-1}\eta_{b+1,i,k}^{(t)}\theta_{b,h,j}^{(r+s-1-t)}\right)\\
   &+   \frac{1}{2}(\beta)^{i+j}(-1)^{r} \left( \delta_{h,i'} \delta_{a,b} \sum_{t=0}^{r-1} \sum_{g=1}^{\nu_{\alpha}}   \eta_{a,j',g}^{(t)}\theta_{a,g,k}^{(r+s-1-t)}-\delta_{a,b+1} \sum_{t=0}^{r-1}\eta_{b+1,j',k}^{(t)}\theta_{b,h,i'}^{(r+s-1-t)}\right)
\end{align*}
Using relations \eqref{yangrel3} and \eqref{yangrel4}.
For $r<s$ we have
\begin{align*}
\{\theta_{a,i,j}^{(r)}, \theta_{a,h,k}^{(s)}\}&= \frac{1}{4}\left( \{e_{a,i,j}^{(r)}, e_{a,h,k}^{(s)}\}- (\beta)^{i+j+h+k}(-1)^{r+s}\{f_{a,k',h'}^{(s)} f_{a,j',i'}^{(r)}\} \right)\\
&+ \frac{1}{4}\left((\beta)^{h+k}(-1)^{s+s_{a,a+1}}\{e_{a,i,j}^{(r)}, f_{a,k',h'}^{(s)}\}- (\beta)^{i+j}(-1)^{r+s_{a,a+1}}\{e_{a,h,k}^{(s)} f_{a,j',i'}^{(r)}\}\right)\\
&=\frac{1}{4}\sum_{t=r}^{s-1} e_{a;i,k}^{(t)} e_{a;h,j}^{(r+s-1-t)} \\
&-\frac{1}{4} (\beta)^{i+j+k+h}(-1)^{r+s}\sum_{t=r}^{s-1} f_{a;k',i'}^{(t)} f_{a;h',j'}^{(r+s-1-t)} \\
&+\frac{1}{4}(\beta)^{h+k}(-1)^{s+s_{a,a+1}} \sum_{t=0}^{r+s-1}
\widetilde{d}_{a;i,h'}^{(r+s-1-t)}d_{a+1;k',j}^{(t)} \\
&-\frac{1}{4}(\beta)^{i+j}(-1)^{r+s_{a,a+1}}\sum_{t=0}^{r+s-1}
\widetilde{d}_{a;h,i'}^{(r+s-1-t)}d_{a+1;j',k}^{(t)} \\
&=\frac{1}{2} \left( \sum_{t=r}^{s-1} \theta_{a,i,k}^{(t)} {\theta}_{a,k,j}^{(r+s-1-t)}+(\beta)^{h+k}(-1)^{s_a+s} \sum_{t=0}^{r+s-1}\widetilde{\eta}_{a,i,h'}^{(r+s-1-t)}\eta_{a+1,k',j}^{(t)} \right)\
\end{align*}
Where we used relations \eqref{yangrel2}, \eqref{yangrel5} and \eqref {yangrel6}. The next relation follows from using relations \eqref{yangrel2}, \eqref{yangrel7} and \eqref{yangrel8}.
\begin{align*}
\{\theta_{a,i,j}^{(r)}, \theta_{a+1,h,k}^{(s+1)}\}-\{\theta_{a,i,j}^{(r+1)}, \theta_{a+1,h,k}^{(s)}\}&=\frac{1}{4}\left(\{e_{a,i,j}^{(r)}, e_{a+1,h,k}^{(s+1)}\}-\{e_{a,i,j}^{(r+1)}, e_{a+1,h,k}^{(s)}\}\right)\\
&+\frac{1}{4}(\beta)^{i+j+k+h}(-1)^{s_{a,a+2}+r+s} \{f_{a,j',i'}^{(r+1)}, f_{a+1,k',h'}^{(s)}\}\\
&-\frac{1}{4}(\beta)^{i+j+k+h}(-1)^{s_{a,a+2}r+s}\{f_{a,j',i'}^{(r)}, f_{a+1,k',h'}^{(s+1)}\} \\
&=-\frac{1}{4}\delta_{h,j}\sum_{g=1}^{\nu_{a+1}}e_{a;i,g}^{(r)} e_{a+1;g,k}^{(s)}\\
&-\frac{1}{4}(\beta)^{i+j+h+k}(-1)^{s_{a,a+2}+r+s}\delta_{h',j'}\sum_{g=1}^{\nu_{a+1}}f_{a+1;k',g}^{(s)}f_{a;g,i'}^{(r)} \\
&=-\frac{1}{2}\delta_{h,j}\left( \sum_{g=1}^{\nu_{a+1}}\theta^{(r)}_{a,i,g}\theta^{(s)}_{a+1,g,k}\right)
\end{align*}
Finally the last relation follows from using relations \eqref{yangrel2}, \eqref{yangrel11}, \eqref{yangrel12}, \eqref{yangrel3}, \eqref{yangrel4} and Lemma~\ref{tildelemma}.
\begin{align*}
    &\{\theta_{a,i,j}^{(r)}, \{\theta_{a,h,k}^{(s)}, \theta_{b,f,g}^{(t)}\}\} +\{\theta_{a,i,j}^{(s)}, \{\theta_{a,h,k}^{(r)}, \theta_{b,f,g}^{(t)}\}\} \\
    &= \frac{1}{8}(\beta)^{h+k+f+g}(-1)^{s_{a,a+1}+s_{b,b+1}+s+t}\{e_{a,i,j}^{(r)}, \{f_{a,k',h'}^{(s)}, f_{b,g',f'}^{(t)}\}\}\\
    &+\frac{1}{8}(\beta)^{h+k+f+g}(-1)^{s_{a,a+1}+s_{b,b+1}+r+t}\{e_{a,i,j}^{(s)}, \{f_{a,k',h'}^{(r)}, f_{b,g',f'}^{(t)}\}\} \\
    &+\frac{1}{8} (\beta)^{i+j}(-1)^{s_{a,a+1}+r}\{f_{a,j',i'}^{(r)}, \{e_{a,h,k}^{(s)}, e_{b,f,g}^{(t)}\}\}\\
    &+\frac{1}{8} (\beta)^{i+j}(-1)^{s_{a,a+1}+s}\{f_{a,j',i'}^{(s)}, \{e_{a,h,k}^{(r)}, e_{b,f,g}^{(t)}\}\}
\end{align*}
\begin{align*}
    &=\frac{1}{8}(\beta)^{h+k+f+g}(-1)^{s+t + s_{a,a+1}+ s_{b,b+1}}(1 + (-1)^{r+s}) \{\{e_{a,i,j}^{(r)}, f_{a,k',h'}^{(s)}\}, f_{b,g',f'}^{(t)}\} \\
    &-\frac{1}{8} (\beta)^{i+j}(-1)^{r+ s_{a,a+1}}(1 + (-1)^{r+s}) \{\{e_{a,h,k}^{(r)}, f_{a,j',i'}^{(s)}\}, e_{b,f,g}^{(t)}\} \\
    &=\frac{1}{8}(\beta)^{h+k+f+g}(-1)^{s+t + s_{a,a+1}+ s_{b,b+1}}(1 + (-1)^{r+s})\left\{ \sum_{t'=0}^{r+s-1} \widetilde{d}_{a;i,h'}^{(r+s-1-t')}d_{a+1;k',j}^{(t')},f_{b,g',f'}^{(t)}\right\}\\
    &-\frac{1}{8}(\beta)^{i+j}(-1)^{r + s_{a,a+1}}(1 + (-1)^{r+s}) \left \{\sum_{t'=0}^{r+s-1}
\widetilde{d}_{a;h,i'}^{(r+s-1-t')}d_{a+1;j',k}^{(t')},e_{b,f,g}^{(t)}\right \}
\end{align*}
If $r+s$ is odd then this vanishes whereas if $r+s$ is even this simplifies too
\begin{align*}
&-\frac{1}{4}(\beta)^{h+k+f+g}(-1)^{s+t + s_{a,a+1} + s_{b,b+1}}\sum_{t'=0}^{r+s-1} \left( \widetilde{d}_{a;i,h'}^{(r+s-1-t')}\delta_{a+1,b} 
\delta_{k',f'}\sum_{t''=0}^{t'-1}\sum_{m=1}^{\nu_{a+1}} 
f_{a+1;g',m}^{(t'+t-1-t'')}d_{a+1;m,j}^{(t'')} \right) \\
&-\frac{1}{4}(\beta)^{h+k+f+g}(-1)^{s+t + s_{a,a+1} + s_{b,b+1}}\sum_{t'=0}^{r+s-1} \left(d_{a+1;k',j}^{(t')} \delta_{a,b+1} \delta_{g,h}\sum_{t''=0}^{r-t'-1} 
\sum_{m=1}^{\nu_a}\widetilde{d}_{a;i,m}^{(t'')} f_{b;m,f'}^{(r-t'+t-1-t'')} \right) \\
&-\frac{1}{4}(\beta)^{i+j}(-1)^{r + s_{a,a+1}}\sum_{t'=0}^{r+s-1} \left(\widetilde{d}_{a;h,i'}^{(r+s-1-t')}\delta_{a+1,b} \delta_{f,k}
\sum_{t''=0}^{t'-1} 
\sum_{m=1}^{\nu_{a+1}}d_{a+1;j',m}^{(t'')} e_{a+1;m,g}^{(t'+t-1-t'')} \right)\\
&-\frac{1}{4}(\beta)^{i+j}(-1)^{r+ s_{a,a+1}}\sum_{t'=0}^{r+s-1} \left(d^{(t')}_{a+1,j',k}\delta_{a,b+1}\delta_{h,g}\sum_{t''=0}^{r-t'-1}\sum_{m=1}^{\nu_a} 
e_{b;f,m}^{(r-t'+t-1-t'')} \widetilde{d}_{a;m,i'}^{(t'')} \right) \\
&=\frac{1}{2}(\beta)^{i+j}(-1)^{r+ s_{a,a+1}+1}\sum_{t'=0}^{r+s-1} \left(\widetilde{\eta}_{a;h,i'}^{(r+s-1-t')}\delta_{a+1,b} \delta_{f,k}
\sum_{t''=0}^{t'-1} 
\sum_{m=1}^{\nu_{a+1}}\eta_{a+1;j',m}^{(t'')} \theta_{a+1;m,g}^{(t'+t-1-t'')} \right)\\
&+\frac{1}{2}(\beta)^{i+j}(-1)^{r+ s_{a,a+1}+1}\sum_{t'=0}^{r+s-1} \left(\eta^{(t')}_{a+1,j',k}\delta_{a,b+1}\delta_{h,g}\sum_{t''=0}^{r-t'-1}\sum_{m=1}^{\nu_a} 
\theta_{b;f,m}^{(r-t'+t-1-t'')} \widetilde{\eta}_{a;m,i'}^{(t'')} \right)
\end{align*}
Let $\widehat{R}(^{\nu}y_n(\sigma), \tau)$ denote the poisson algebra with generators \eqref{diryanggens} and relations \eqref{diryangrel0}-\eqref{diryangrel7}.  We have shown that there is a surjection $\widehat{R}(^{\nu}y_n(\sigma), \tau) \twoheadrightarrow {R}(^{\nu}y_n(\sigma), \tau) $. To conclude the proof we show that it sends a spanning set to a basis. We define a loop filtration on $\widehat{R}(^{\nu}y_n(\sigma), \tau)=\bigcup_{i \geq 0}F_{i}\widehat{R}(^{\nu}y_n(\sigma), \tau)$ by placing the generators $\theta_{a,i,j}^{(r)}, \eta_{a,i,j}^{(r)}$ in degree $r-1$ and the bracket in degree 0.

Examining the top filtered components of \eqref{diryangrel1}-\eqref{diryangrel7} we see  from Theorem \ref{dirredcurrent} that there is a surjective Poisson morphism $R(S(\cc_n(\sigma), \tau) \twoheadrightarrow \gr \widehat{R}(^{\nu}y_n(\sigma), \tau) $. We thus see that $\gr \widehat{R}(^{\nu}y_n(\sigma), \tau)$ is generated as a commutative algebra by the elements
$\eta_{a,i,j}^{(r)}+F_{r-2}\widehat{R}(^{\nu}y_n(\sigma), \tau),\theta_{a,i,b,j}^{(r)}+F_{r-2}\widehat{R}(^{\nu}y_n(\sigma), \tau)$ which are again defined via the recursion \eqref{e:higherthetas}.

By a standard filtered algebra argument $ \widehat{R}(^{\nu}y_n(\sigma), \tau)$ is generated as a commutative algebra by the elements with the same names in $\gr \widehat{R}(^{\nu}y_n(\sigma), \tau)$.
 We have thus shown that $\widehat{R}(^{\nu}y_n(\sigma), \tau) \twoheadrightarrow {R}(^{\nu}y_n(\sigma), \tau) $ maps a spanning set to a basis, which concludes the proof. 
\end{proof}
\begin{Remark}
Strictly speaking we need to take care of the linear relations \eqref{diryangrel0} between the $\eta_{a,i,j}^{(r)}$ to conclude that a spanning set gets sent to a basis.
This can be dealt with in exactly the same way as the proof of Theorem \ref{T:twistedcurrentalgpresenrtation}.
\end{Remark}
\begin{Remark}
Note that \cite[Theorem~3.7]{To23} is a special case of this presentation where $\nu=(1,1,\dots,1,1)$ the only difference being in scaling  due to a different convention of the scaling of generators as explained in Remark \ref{scaling}.
\end{Remark}

Let $\gr R(y_n(\sigma), \tau)$ denote the graded algebra for the loop filtration.
In the last paragraph of the proof of Theorem~\ref{dirredyangpres} we obtained the following important result.
\begin{Corollary}
\label{C:loopforcurrent}
If $\sigma$ is a symmetric shift matrix and $^{\nu}y_n(\sigma)$ is the semiclassical shifted Yangian then there is an isomorphism of commutative algebras $S(^{\nu}\cc_n(\sigma)^\tau) \isoto \gr R(^{\nu}y_n(\sigma), \tau)$ given by
\begin{eqnarray}
\label{e:PDloopiso}
\begin{array}{lcl}
\theta_{a,i,j;r-1} \longmapsto \theta_{a,i,j}^{(r)} + \F_{r-2}y_n(\sigma) & \text{ for  all admissible } a,i,j, \ s_{a,a+1}(\nu)<r \\
\eta_{a,i,j;r-1} \longmapsto \eta_{a,i,j}^{(r)} + \F_{r-2}y_n(\sigma) & \text{ for  all admissible } a,i,j, \, 0 < r.\vspace{-18pt}
\end{array}
\end{eqnarray}
$\hfill \qed$
\end{Corollary}

\section{The Dirac reduced semiclassical Brundan--Kleshchev homomorphism}

\subsection{The semiclassical Brundan--Kleshchev homorphism}

Fix $\ve = \pm 1$ and an integer $N > 0$ such that $\ve^N = 1$. Let 
$\lambda = (\lambda_1 \le \lambda_2 \le \cdots \le \lambda_n) \in \P_\ve(N)$ be a partition associated with all parts of the same parity. Following Section~\ref{ss:Dynkinandcentraliser} we choose a nilpotent element $e\in \gl_N$ with Jordan blocks of size $\lambda_1,...,\lambda_n$, and an involution $\tau : \gl_N \to \gl_N$ such that the fixed point subalgebra $\g = \gl_N^\tau$ is simple of type
$$\begin{array}{ccc}{\sf C} & \text{ when } &\ve = -1\\ 
{\sf B} \text{ or } {\sf D} & \text{ when } &\ve = 1. \end{array}$$
Since $\tau$ fixes the $\sl_2$-triple $(e,h,f)$ it also stabilises the slice $e + \gl_N^f$ and acts by Poisson automorphisms on $S(\gl_N,e)$.

Now let $(\sigma, \ell)$ be the $n\times n$ shift matrix and the level determined by $\lambda$ as in \eqref{L:lambdabij}. We pick a $(-\ve(-1)^{\lambda_1}, \sigma)$-admissible shape and by slight abuse of notation we also let $\tau$ denote the automorphism of $^\nu y_n(\sigma)$ described in \eqref{e:tauon}.
\begin{Theorem}
\label{T:BKhom}
There is a $\tau$-equivariant surjective Poisson homomorphism $^\nu y_n(\sigma) \to S(\g,e)$ determined by
\begin{eqnarray*}
\begin{array}{rcl}
d_{a, i, j}^{(r)} & \longmapsto & t_{a, i, a, j, \bnu_{a-1}}^{(r)}\\
e_{a,i, j}^{(r)} & \longmapsto &  t_{a, i, a+1, j, \bnu_{a}}^{(r)}\\
f_{a,i, j}^{(r)} & \longmapsto &  t_{a+1, i, a, j, \bnu_{a}}^{(r)}.
\end{array}
\end{eqnarray*}
The kernel is generated as a Poisson ideal by
\begin{eqnarray}
\label{e:kernelgensA}
\{d_{1, i, j}^{(r)}\mid 1\le i,j \le \nu_a, \ r > \lambda_1\}.
\end{eqnarray}

\end{Theorem}
\begin{proof}
The existence of the homomorphism from the shifted Yangian to the quantum finite $W$-algebra is proven by Brundan and Kleshchev in \cite[Theorem~10.1]{BK06}. The result stated here follows from theirs by taking the associated graded algebras with respect to the canonical filtration and the Kazhdan filtration. The $\tau$-equivariance follows by comparing Proposition~\ref{P:tauequivariance} with formula \eqref{e:tauon}.

The fact that the elements \eqref{e:kernelgensA} lie in the kernel of the map is an immediate consequence of \cite[Theorem~10.1]{BK06}. The fact that these elements generate the kernel can be proven using an argument identical to \cite[Proposition~4.7]{To23}, deploying our Lemma~\ref{L:shiftedcurrentandcentraliser}(1).
\end{proof}

\subsection{Dirac reduction of the homomorphism}

We now proceed to introduce an element of $y_n(\sigma)^\tau$ which lies in the kernel of the map appearing in Theorem~\ref{T:BKhom}, which deforms the element \eqref{e:extrakernelLA} via the loop filtration.

To do so first let $\od_{1,i,j}^{(\lambda_1(\nu)+1)} := \cd_{1,i,j}^{(\lambda_1(\nu)+1)} \in \ ^{\nu}y_n(\sigma)$ 
Note that  $\tau(\od_{1,i,j}^{(\lambda_1(\nu)+1)}) = -\od_{1,i,j}^{(\lambda_1(\nu)+1)}.$
Therefore we may define
\begin{eqnarray}
\label{e:additionalgenerators}
\ttheta_{1,i,j}^{(\lambda_1(\nu) + s_{1,2} + 1)} := \{\od_{1,i,k}^{(\lambda_1(\nu)+1)}, \ce_{1,k,j}^{(s_{1,2}(\nu) + 1)}\} + \cy_n(\sigma)^\tau \in y_n(\sigma)^\tau/\cy_n(\sigma)^\tau = R(y_n(\sigma), \tau)
\end{eqnarray}
\begin{Theorem}
\label{kernelslicepres}
There is a surjective Poisson homomorphism $\varphi:R(^\nu y_n(\sigma), \tau) \onto S(\g_0, e)$. 
\begin{itemize}
\setlength{\itemsep}{4pt}
\item[(i)] When $\ve=1$, the kernel is Poisson generated  by 
\begin{eqnarray}
\label{e:lambdaoddgens}
\{\eta_{1,i,j}^{(r)} \mid r > \lambda_1(\nu) \ , 1 \le i,j \le \nu_1\}.
\end{eqnarray}

\item[(ii)] When $\ve=-1$ and $\nu_1 \ne 1$ the kernel is Poisson generated by \eqref{e:lambdaoddgens}

\item[(iii)] When $\ve= -1$ and $\nu_1 = 1$ the kernel is Poisson generated by those same elements, along with 
\begin{eqnarray}
\label{e:extrakernel}
\{\ttheta_{1,i,j}^{( \lambda_1(\nu) + s_{1,2}(\nu)+1)} \mid  1\le i , \le \nu_1 \ , 1 \le j \le \nu_{2} \}.
\end{eqnarray}
\end{itemize}
\end{Theorem}
\begin{proof}
The existence of this surjective map follows by \cite[(2.6) \&Theorem 2.6]{To23} and Theorem~\ref{T:BKhom}.
For the kernel consider first the case (i).
It follows from Theorem \ref{T:BKhom} that the elements $\{\eta_{1,i,j}^{(r)} \mid r > \lambda_1 \ , 1 \le i,j \le \nu_1\}$ lie in the kernel of $\varphi$. Let $I$ denote the Poisson ideal generated by \ref{e:lambdaoddgens}. By Lemma \ref{L:shiftedcurrentandcentraliser}(2)(i) we see that the associated graded with respect to the loop filtration $\gr(R(^\nu y_n(\sigma), \tau)/I)$ is generated as a commutative algebra by the top filtered component of the elements
\begin{equation}
\label{Walggens}
\begin{split}
&\{\theta_{a,b,i,j}^{(r)} \mid 1 \leq a <b \leq m, \ 1 \leq i \leq \nu_a, \ 1 \leq j \leq \nu_b, \ s_{a,b}(\nu)<r \leq s_{a,b}(\nu)+\lambda_a(\nu)  \} \\
& \ \ \ \ \ \ \ \ \ \ \ \ \ \cup \{\eta_{a,i,j}^{(r)} \mid 1 \leq a \leq m, \ 1 \leq i,j \leq \nu_a, \ 0<r \leq \lambda_a(\nu)\}
\end{split}
\end{equation}
modulo the linear relation \eqref{diryangrel0}.
In particular $R(^\nu y_n(\sigma), \tau)/I$ is generated by these elements as a commutative algebra.

By \cite[Theorem 2.4]{To23} we have that $S(\g_0, e)$ is a polynomial algebra modulo the same linear relation generated by the images of these elements under $\varphi$ and hence $\varphi$ is an isomorphism.

Now for cases (ii) and (iii) it follows by Theorem \ref{T:BKhom} that the elements \ref{e:extrakernel} lie in the kernel of $\varphi$. Again let $I$ denote the poisson ideal generated by the elements \eqref{e:lambdaoddgens} and \eqref{e:extrakernel}. As before it is enough to show that $R(^\nu y_n(\sigma), \tau)/I$ is generated as a commutative algebra by the elements \eqref{Walggens} modulo \eqref{diryangrel0}. A calculation in $^\nu y_n(\sigma)$ using \eqref{C:loopforcurrent} shows that the top filtered component of $\ttheta_{1,i,j}^{( \lambda_1(\nu) + s_{1,2}(\nu)+1)}$ with respect to the loop filtration is $\theta_{1,i,j; \lambda_1(\nu) + s_{1,2}(\nu)}$.

Using Lemma \ref{L:shiftedcurrentandcentraliser}(2)(ii)\&(iii) we see that the associated graded with respect to the loop filtration $\gr(R(^\nu y_n(\sigma), \tau)/I)$ is generated as a commutative algebra by the top filtered component of the elements \eqref{Walggens} modulo \eqref{diryangrel0}. The proof now proceeds as in case (i).
\end{proof}

\begin{Remark}
This paper was built upon the developments made in \cite[Part I]{To23}. When writing this work we noticed several minor typos in that paper and we list them here for completeness.
\begin{enumerate}
\item There is a sign error in \cite[(3.29), (3.31)]{To23}.

\item  In  \cite[(3.70)]{To23} $\varpi_{r,s}$  should be replaced with $\varpi_{s,r}$
\item Definition 3.22 makes no sense the way it is written; the $f_{i,j}$ are defined for $j > i$, however the recurrence \cite[(2.19)]{BK06} should be followed.

\item The proof of \cite[Lemma 4.2]{To23} has a double sign error, but these cancel and stated result is correct.

\item The map in \cite[(4.18)]{To23} should read $c_{i,i}^{(r - 1)} \mapsto (-1)^{r-1}  t^{(r)}_{i,i,i-1}$.

\end{enumerate}
\end{Remark}

\vspace{10pt}

\noindent {Contact details:}\medskip

Lukas Tappeiner {\sf lt862$@$bath.ac.uk};

Lewis Topley {\sf lt803$@$bath.ac.uk}.\medskip

Department of Mathematical Sciences, University of Bath, Claverton Down, Bath BA2 7AY, United Kingdom.

\end{document}